\documentclass[12pt]{amsart}

\usepackage{amssymb,amsthm,amsmath,amsfonts,wrapfig,enumerate,multicol,tikz-cd,tikz,caption,bm}

\usepackage[numbers]{natbib}

\tikzset{>=latex}

\newtheorem{theorem}{Theorem}[section]

\newtheorem{lemma}[theorem]{Lemma}

\newtheorem{prop}[theorem]{Proposition}
\newtheorem{corollary}[theorem]{Corollary}

\newtheorem{Q}[theorem]{Question}

\newtheorem{innercustomthm}{Theorem}
\newenvironment{customthm}[1]
  {\renewcommand\theinnercustomthm{#1}\innercustomthm}
  {\endinnercustomthm}

\theoremstyle{definition}

\newtheorem{definition}[theorem]{Definition}

\DeclareMathOperator{\Ric}{Ric}
\DeclareMathOperator{\K}{K}
\DeclareMathOperator{\2}{II}

\DeclareMathOperator{\RP}{\mathbf{R}P}
\DeclareMathOperator{\CP}{\mathbf{C}P}
\DeclareMathOperator{\HP}{\mathbf{H}P}
\DeclareMathOperator{\OP}{\mathbf{O}P}
\DeclareMathOperator{\Ima }{im}
\DeclareMathOperator{\ID }{id}

\usepackage{fullpage}

\usepackage{hyperref}

\usepackage{amsmath,amssymb,amsfonts,wrapfig,enumerate,multicol,tikz-cd,tikz,caption,bm} 
\usepackage[numbers]{natbib}

\title{Ricci-Positive Metrics on Connected Sums of Projective Spaces}
\author{Bradley Lewis Burdick}
\address{Department of Mathematics\\
University of Oregon \\
Eugene, OR, 97405\\
USA
}

\begin{document}

\maketitle

\begin{abstract}
It is a well known result of Gromov that all manifolds of a given dimension with positive sectional curvature are subject to a universal bound on the sum of their Betti numbers. On the other hand, there is no such bound for manifolds with positive Ricci curvature: indeed, Perelman constructed positive Ricci metrics on arbitrary connected sums of complex projective planes. In this paper, we revisit and extend Perelman's techniques to construct positive Ricci metrics on arbitrary connected sums of complex, quaternionic, and octonionic projective spaces in every dimension. 

\noindent MSC classes: 53C20, 53C25.
\end{abstract}

\tableofcontents



\section{Introduction} \label{introduction}



\subsection{Background and Main Results}


The earliest examples of manifolds which admit metrics with positive Ricci curvature but no positively curved metrics are due to Sha and Yang \cite{Sha1,Sha}. Soon after Perelman constructed new examples of such spaces. 

\begin{theorem}\cite{Per1}\label{perelmanresult} For all $k>0$, the manifolds $\#_k \CP^2$ admit metrics with positive Ricci curvature. 
\end{theorem} 
\noindent In this paper, we show that Perelman's claim remains true for all complex, quaternionic, and octonionic projective spaces. 

\begin{customthm}{A}\label{B} For all $k>0$ and $n>0$ the manifolds $\#_k\CP^n$, $\#_k\HP^n$, and $\#_k\OP^2$ admit metrics with positive Ricci curvature.  
\end{customthm}

The approach to proving Theorem \ref{perelmanresult} in \cite{Per1} hinges on a gluing lemma for positive Ricci curvature (Lemma \ref{glue}). Given two Riemannian manifolds with positive Ricci curvature and isometric boundaries, the lemma gives a sufficient condition to glue these manifolds together along their boundaries while preserving positive Ricci curvature.
\begin{lemma}\cite[Section 4]{Per1}\label{glue}
Let $(M^n_i,g_i)$ with $i=1,$ $2$ and $n\ge 3$ be closed Riemannian manifolds with positive Ricci curvature. Suppose that there is an isometry $\phi$ between their boundaries such that $\2_1+\phi^*\2_2$ is a positive definite 2-tensor on $M_1$, where $\2_i$ is the second fundamental form of $\partial M_i$ with respect to the outward normal. Then $M_1\cup_\phi M_2${ admits a metric with positive Ricci curvature that agrees with }$g_i${ on }$M_i${ outside of an arbitrarily small tubular neighborhood of }$\partial M_i$. 
\end{lemma}

Of Perelman's technical results in \cite{Per1}, Lemma \ref{glue} has attracted the most attention. It was used in \cite{AMW} with Hamilton's work on Ricci flow to prove that the space of Ricci positive metrics on $D^3$ with convex boundary is path connected. In \cite{BWW}, the gluing Lemma is generalized to glue together compact families of Riemannian metrics to demonstrate nontrivial homotopy groups of the observer moduli space of metrics of positive Ricci curvature. The proof of Lemma \ref{glue} is sketched in \cite[Section 4]{Per1}, a special case of Lemma \ref{glue} is proven in \cite[Lemma 2.3]{AMW}, and very detailed proofs are presented in \cite[Lemma 2.3]{Wa} and \cite[Section 2]{BWW}. We refer the reader to these sources for details and computations. 

 Lemma \ref{glue} allows us to construct metrics with positive Ricci curvature by breaking our manifolds up into pieces and constructing positive Ricci metrics with compatible boundaries. 
Using Lemma \ref{glue} Perelman divided the construction of Ricci positive metrics on $\#_k\CP^2$ into two main pieces:
\begin{enumerate}[(i)]
\item the construction of a Ricci positive metric on $S^n_k$, the sphere with $k$ disjoint balls removed, with relatively small principal curvatures of the boundary, which we call \emph{the docking station},
\item the construction of Ricci positive metrics on $\CP^2\setminus D^4$ with positive principal curvatures of the boundary, which Perelman called the cores. 
\end{enumerate}

The construction of the docking station makes up the bulk of \cite{Per1} as it is itself made out of two technical constructions summarized below as Lemmas \ref{actualambient} and \ref{neck}. We will explain how Proposition \ref{docking} follows from these two lemmas and Lemma \ref{glue} in Section \ref{dockingsection}. 

\begin{prop}\label{docking} For all $n>3$, $k>0$, and $\rho<1$ there exists a metric $g_\text{docking}$ on $S^n_k$ with positive Ricci curvature, such that the metric restricted to each boundary component is round with radius $\rho$, and such that the second fundamental form of the boundary with respect to the outward normal is $\2=-g_\text{docking}$. 
\end{prop} 

Once the docking station of Proposition \ref{docking} is constructed, Lemma \ref{glue} reduces the existence of a Ricci positive metric on $\#_{i=1}^k M^n_i$ to constructing Ricci positive metrics on each of $M^n_i\setminus D^n$ with round, convex boundaries. 
 
 \begin{customthm}{B}\label{parts} Let $n\ge 4$ and $1\le i\le k$. Suppose there exists metrics $g_i$ on $M^n_i\setminus D^n$ with positive Ricci curvature such that the second fundamental form $\2_i$ of the boundary of $M^n_i\setminus D^n$ is positive definite and the metric restricted to the boundary is round. Then there exists a metric on $\#_{i=1}^k M_i$ with positive Ricci curvature. 
\end{customthm}

\noindent In section \ref{assembly}, we explain in detail how Theorem \ref{parts} follows from Proposition \ref{docking}. 

A metric $g_i$ on $M_i^n\setminus D^n$ satisfying the hypotheses of Theorem \ref{B} will be called a \emph{core metric for $M$}. The combined work of \cite[Section 2]{Per1} and \cite{Per2} was to construct core metrics for $\CP^2$, thus completing the proof of Theorem \ref{perelmanresult}. Sections \ref{doublesub} and \ref{core} are dedicated to constructing analogous core metrics for all projective spaces. 

\begin{customthm}{C}\label{A} There are metrics $g_\text{core}$ on each of $\CP^n\setminus D^{2n}$, $\HP^n\setminus D^{4n}$, and $\OP^2\setminus D^{16}$ satisfying the hypotheses of Theorem \ref{parts} .
\end{customthm}

\noindent Theorem \ref{B} then follows from Theorem \ref{parts} and Theorem \ref{A}. 

Perelman's construction of core metric for $\CP^2$ consider $\CP^2~\setminus~D^4$ as the normal disk bundle of an embedded $\CP^1 \hookrightarrow \CP^2$. The normal sphere bundle of this embedding is the Hopf fibration, hence its total space is diffeomorphic to $S^3$. One can therefore think of $\CP^2\setminus D^4$ as a quotient of $I\times S^3$. Using the left invariant framing given by the Lie group structure on $S^3$, one can define a warped product metric on $I\times S^3$ which under certain conditions placed on the warping functions will descend to a metric on $\CP^2\setminus D^4$. This is the class of metrics that Perelman considered in \cite[Section 2]{Per1}, and in \cite{Per2} he used the Lie bracket to compute the curvature of such metrics. He then showed that this metric will satisfy Theorem \ref{A} for a careful choice of warping functions. Clearly this construction and method of computation will fail in higher dimensions. 

To generalize Perelman's approach one may consider a general projective space $\mathbf{P}^n$ either over $\mathbf{C}$, $\mathbf{H}$, or $\mathbf{O}$ (where $n\le 2$ for $\mathbf{O}$, see \cite[Corollary 4L.10]{Ha}), and let $d$ denote the real dimension of the underlying algebra. We can again consider $\mathbf{P}^{n}\setminus D^{dn}$ as the normal disk bundle of an embedding $\mathbf{P}^{n-1}\hookrightarrow\mathbf{P}^{n}$. The normal sphere bundle can then be identified with the generalized Hopf fibration $S^{dn-1}\rightarrow \mathbf{P}^{n-1}$. One can again think of $\mathbf{P}^{n}\setminus D^{dn}$ as a quotient of $I\times S^{dn-1}$. While no global framing will exist, the tangent space to any fiber bundle will still admit a global decomposition into vertical and horizontal subbundles. This split is sufficient to define a metric analogous to the doubly warped product, belonging to a class which we call \emph{doubly warped Riemannian submersion metrics}. 

The precise construction of doubly warped Riemannian submersion metrics and their properties occupies Section \ref{doublesub}. We then apply doubly warped Riemannian submersion metrics to construct metrics on $\mathbf{P}^n\setminus D^{dn}$ as outlined here and prove Theorem \ref{A} in Section \ref{core} .  

Once establishing the existence of core metrics for $\mathbf{P}^n$, we are able to combine the ideas of \cite{Per1} with \cite{Sha} and give a generalization of the following replacing $S^n$ with $\mathbf{P}^n$.  

\begin{theorem}\cite[Theorem 1]{Sha}\label{sharesult} For all $k \ge0$, $n\ge 2$, and $m\ge 2$, the manifolds $\#_k(S^n\times S^m)$ admit metrics with positive Ricci curvature.
\end{theorem}

\noindent The approach to proving Theorem \ref{sharesult} in \cite{Sha} is to prove a surgery theorem for positive Ricci curvature (see \cite[Lemma 1]{Sha}). Then observing that $\#_k S^n\times S^m$ can be constructed out of $S^{n-1}\times S^{m+1}$ by performing surgery $(k+1)$ times on embedded $S^{n-1}$, the claim follows. In section \ref{productconnect}, the technical work of \cite{Per1} (specifically Lemma \ref{neck} below) can be viewed as a modified surgery result: allowing under certain conditions to remove an $S^{n-1}\times D^{m+1}$ and attach an $(N^n\setminus D^n )\times S^m$ while preserving positive Ricci curvature. Arguing as in \cite{Sha}, Proposition \ref{decomp} claims that $\#_k(N^n\times S^m) $ can be constructed out of $S^{n-1}\times D^{m+1}$ by performing this modified surgery $(k+1)$ times. Applying this construction to projective spaces allows us to form Ricci positive metrics on connected sums of $\mathbf{P}^n\times S^m$. 

To summarize, the full list of spaces for which we have constructed Ricci positive metrics is as follows.  

\begin{customthm}{$\text{A}^\prime$}\label{realproj} For $\sigma\in \{0,1\}$, $n\ge 1$, $m\ge 3$, and $i,j,k,l \ge 0$, the following manifolds admit metrics with positive Ricci curvature.
\end{customthm}
\begin{enumerate}[(i)]
\item $\left(\#_\sigma \RP^{2n}\right)\#\left(\#_j\CP^n\right)$
\item $\left(\#_\sigma \RP^{4n}\right)\#\left(\#_j\CP^{2n}\right) \#\left( \#_k \HP^n\right)$
\item $\left(\#_\sigma \RP^{16}\right)\#\left(\#_j\CP^{8}\right)\#\left(\#_k \HP^{4}\right)\#\left(\#_l\OP^2\right)$
\item $(\#_i S^{2n}\times S^m) \#( \#_j \CP^n\times S^m) $
\item $(\#_i S^{4n}\times S^m) \# (\#_j \CP^{2n}\times S^m) \# ( \#_k \HP^n\times S^m)$
\item $(\#_i S^{16}\times S^m) \#(\#_j \CP^8 \times S^m) \# (\#_k \HP^4 \times S^m) \# (\#_l \OP^2 \times S^m)$
\end{enumerate}

The addition of $\RP^n$ to the list of spaces in Theorem \ref{realproj} is made possible in Section \ref{lens} by the observation that the part of the construction of the docking station from Lemma \ref{actualambient} has isometry group with many finite subgroups. Corollary \ref{lensdocking} describes how this allows us to construct analogous metrics to those in Proposition \ref{docking} so that the docking station can be taken to be certain finite quotients of $S^n$, in particular $\RP^n$ can play the role of the docking station from which (i), (ii), and (iii) of Theorem \ref{realproj} follow.

Theorem \ref{realproj} represents infinitely many new examples of manifolds that admit metrics with positive Ricci curvature but admit no metrics that are positively curved (as mentioned in the abstract, this follows from \cite{Gr2}). In general, other than the examples already mentioned  and the resolution of the Calabi Conjecture there are few examples of manifold with positive Ricci curvature. 


\subsection{Notation}

We will always assume that manifolds are compact and smooth (possibly with boundary). For $k\ge 0$ will denote by $\#_k M^n$ the $k$-fold connected sum of $M$ with itself, where by convention $\#_0 M^n = S^n$. 

Throughout this paper $\mathbf{P}^n$ will denote projective space over any real division algebra, and $d$ will denote the real dimension of the underlying algebra, where $n\le 2$ for octonions, and $S^n_k$ will denote $S^n$ with $k$ disks removed. 

We will write $(M,g)$ to denote a smooth manifold with smooth Riemannian metric. We will use subscripts on curvatures to indicate the metric they are derived from, i.e. $\Ric_g(X,Y)$ is the Ricci curvature associated to the Levi-Civita connection of $g$. The second fundamental form of a submanifold will be denoted by $\2$, which when that submanifold is the boundary will be considered as $(0,2)$-form by pairing with outward unit normal. The boundary is called convex or concave if $\2$ is positive or negative definite respectively.


\section{Doubly Warped Riemannian Submersion Metrics} \label{doublesub}


Doubly warped Riemannian submersion metrics generalize doubly warped product metrics to the case of nontrivial Riemannian submersion, which we consider in general in Section \ref{background}. Once establishing basic notation and facts, we give the precise definition of doubly warped Riemannian submersion metrics in Section \ref{doublesub}, and we record their Ricci curvatures in a particular case in Lemma \ref{wraithricci}. Finally in Section \ref{quotient}, we explain how doubly warped Riemannian submersion metrics can be used to define metrics on disk bundles of Euclidean vector bundles. 

A larger class of doubly warped Riemannian submersion metrics have been studied previously in \cite{Wr3} to show an explicit lower bounds on the dimension of stabilization required to construct Ricci positive metrics on fiber bundles where both fiber and base support metrics with nonnegative Ricci curvature. They have also been used in \cite[Example 3]{Che2} to construct metrics of nonnegative curvature on $\mathbf{P}^n\#\mathbf{P}^n$. See \cite[Section 1.4.6]{Pet} for exposition for such metrics in the context of the Hopf fibrations.


\subsection{Riemannian Submersions}\label{background}


Our notation for Riemannian submersions largely agrees with \cite[Chapter 9]{Be}; we refer the reader there for a full account of the subject. We will assume that $\pi: E^{n+m}\rightarrow B^m$ is a smooth surjective submersion, i.e. a surjective smooth map with differential everywhere onto. We now record some notations and terminology for the structure associated to any smooth surjective submersions.\\

\begin{figure}
\centering
\begin{tikzpicture}[scale=.21]
  \node (img)  {\includegraphics[scale=0.21]{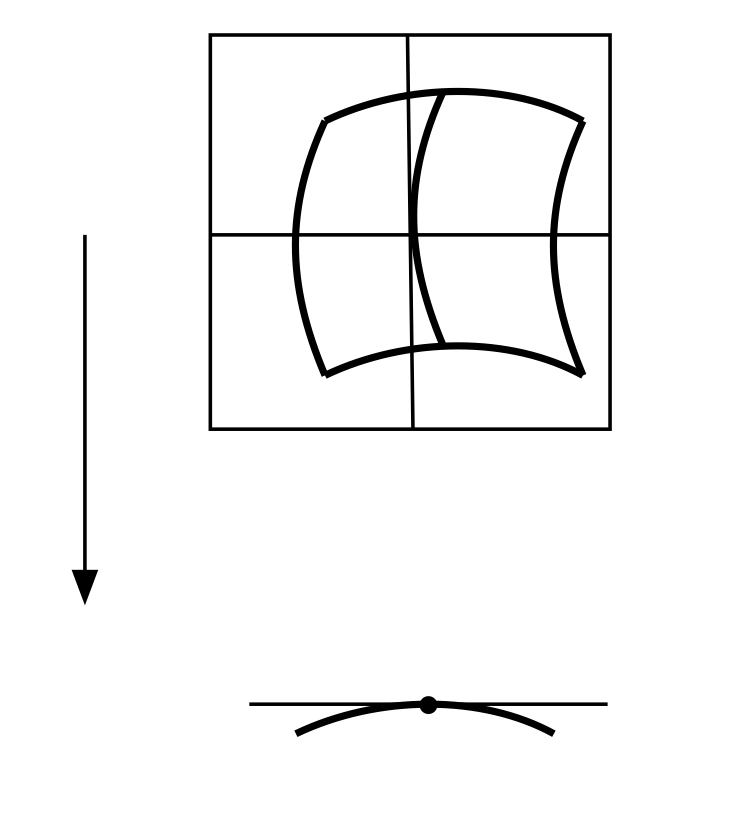}};
\node[left] at (1.8,12) {$V_x$};
\node[right] at (2,10) {$F_b$};
\node[below left] at (7.5,10.5) {$E$};
\node[right] at (-5.8,5.2) {$H_x$};
\node[left] at (1.8,6.8) {$x$};
\node[left] at (8,.3) {$T_xE$};
\node[above] at (2,-10.3) {$b$};
\node[left] at (-2,-12) {$B$};
\node[left] at (9,-9) {$T_bB$};
\node[right] at (-9.5, 0) {$\pi$};
 \end{tikzpicture}
 \caption{Schematic of Riemannian submersion showing notation of Definition \ref{submersionfacts}.}
 \label{fig:riemsub}
\end{figure}
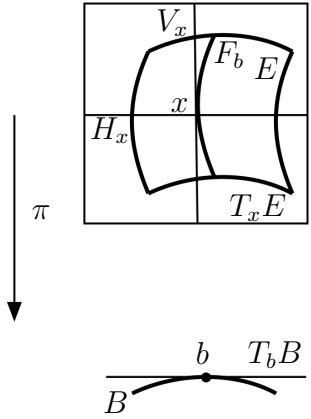

\begin{definition}\cite[9.7 and 9.8]{Be}\label{submersionfacts} For any smooth surjective submersion $\pi: (E^{n+m},g)\rightarrow (B^m,\check{g})$, define as in Figure \ref{fig:riemsub}
\begin{enumerate}[i.]
\item The fiber over $b\in B^m$, $F_b^n:= \pi^{-1}(b)$;
\item The fiber metric $\hat{g}_b:= g|_{F_b^n}$;
\item The vertical distribution $V:=\ker d\pi \subseteq TE^{n+m}$;
\item The horizontal distribution $H:=V^\perp$, where this compliment is taken with respect to $g$;
\item The vertical and horizontal projections $\mathcal{V}:TE^{n+m}\rightarrow V$ and $;\mathcal{H}:TE\rightarrow H$, are the bundle maps with kernels $H$ and $V$ respectively.
\end{enumerate}
Sections of $V$ and $H$ are called \emph{vertical} and \emph{horizontal} vector fields respectively. 
\end{definition}
A smooth surjective submersion $\pi:(E^{n+m},g)\rightarrow (B^m,\check{g})$ is called a \emph{Riemannian submersion} if for all horizontal vectors $Y_i$ we have $\check{g} (\check{Y}_i,\check{Y}_j)= g(Y_i,Y_j)$, where $\check{Y}:= \pi_* Y$. In other words, $\pi$ is a Riemannian submersion if $\pi_*: H_x \rightarrow T_{\pi(x)}B$ is an isometry. In this paper, we will denote by $Y_i$ and $V_i$ elements of an orthonormal frame of $E^{n+m}$ with respect to $g$ that are horizontal and vertical respectively. 

It is a fact that a Riemannian submersion between two compact manifolds is a smooth fiber bundle with fiber diffeomorphic to a compact manifold $F^n$ (see \cite[Theorem 9.3]{Be}).Thus one can define a metric on the total space of a fiber bundle making the projection $\pi:E^{n+m}\rightarrow B^m$ a Riemannian submersion by specifying the data of Definition \ref{submersionfacts}.

\begin{prop}\cite[9.15]{Be} \label{constructsub} Let $\pi:E^{n+m}\rightarrow B^m$ be fiber bundle with fiber $F^n$. Suppose we are given the following data:
\begin{enumerate}[i.]
\item A horizontal distribution $H\subseteq TE^{n+m}$;
\item A base metric $\check{g}$ on $B^m$;
\item A family of fiber metrics $\hat{g}_b$ on $F_b^n$ parameterized by $b\in B^m$.
\end{enumerate}
Then there is a unique metric $g$ on $E^{n+m}$ that makes $\pi:(E^{n+m},g)\rightarrow (B^m,\check{g})$ a Riemannian submersion with horizontal distribution $H$ and so that $g$ restricted to $F_b^n$ is isometric to $\hat{g}_b$. Moreover, $g$ takes the following form
$$g= \mathcal{H}^* \pi^*\check{g} +\mathcal{V}^* \hat{g},$$
where $\hat{g}$ is the fiber-wise metric defined to by $\hat{g}_b$.
\end{prop}

In our application, the fibrations we will be considering, the Hopf fibrations, are sphere bundles of Euclidean vector bundles. For such fibrations there are preferred choices in the above construction. 

\begin{prop}\cite[9.60]{Be}\label{Vilms} Given a rank $(n+1)$ Euclidean vector bundle $E$ over a Riemannian manifold $(B^m,\check{g})$ with bundle metric $\mu$ and metric connection $\nabla$, $\mu$ determines a $O(n+1)$ invariant fiberwise Riemannian metric $\hat{g}_b$ and $\nabla$ determines a $O(n+1)$ invariant horizontal distribution $H$. Given such data there is a unique metric $g$ on $S(E)$ which makes $\pi:(S(E),g)\rightarrow (B^m,\check{g})$, a Riemannian submersion with horizontal distribution $H$ and totally geodesic fibers isometric to $(S^{n-1},\hat{g}|_{S^{n-1}})$ given by
$$g= \mathcal{H}^*\pi^* \check{g} + \mathcal{V}^* \hat{g},$$
as in Proposition \ref{constructsub}.
\end{prop}


\subsection{Doubly Warped Riemannian Submersions}\label{definitionof}


In this section we introduce doubly warped Riemannian submersion metrics. These generalize doubly warped product metrics on $I\times B^m\times F^n$ by replacing the product $B^m\times F^n$ with nontrivial bundle $\pi: E^{n+m}\rightarrow B^m$ with fiber $F^n$ by scaling fiber and base by functions $f$ and $h$ (see Figure \ref{fig:doublefig}). 

\begin{definition}\label{doublesubdef} Suppose we are given a Riemannian submersion $\pi: (E^{n+m},g)\rightarrow( B^m,\check{g})$, and two positive functions $f(t)$ and $g(t)$ defined for $t\in I$. The \emph{doubly warped Riemannian submersion metric} $\tilde{g}$ on $I\times E$, is a metric so that $\ID \times \pi : (I\times E^{n+m},\tilde{g}) \rightarrow (I\times B^m,\check{\tilde{g}})$ is the Riemannian submersion as in Proposition \ref{constructsub} specified by the following data.
\begin{enumerate}[i.]
\item horizontal distribution $\tilde{H} = TI \oplus H,$
\item base metric $\check{\tilde{g}} = dt^2 + h^2(t) \check{g},$
\item fiber metrics $\hat{\tilde{g}}_{(t,b)} = f^2(t) \hat{g}_b,$
\end{enumerate}
Let $\mathcal{H}:T(I\times E)\rightarrow 0\oplus H$ be the bundle map with kernel $TI\oplus V$, then $\tilde{g}$ takes the following form
\begin{align}
\label{2warpedmetricform} \tilde{g} = dt^2 + h^2(t)\mathcal{H}^*\pi^* \check{g}+ f^2(t) \mathcal{V}^*\hat{g}.
\end{align}
\end{definition}

In this paper $\partial_t$, $X_i$, and $U_i$ will denote elements of an orthonormal frame of $I\times E$ with respect to doubly warped Riemannian submersion metric $\tilde{g}$ that are respectively tangent to $I$, $0\oplus H$, and $V$. Note that the vector fields $X_i$ and $U_i$ of $\tilde{g}$ correspond to $Y_i$ and $V_i$ of $g$ respectively. 

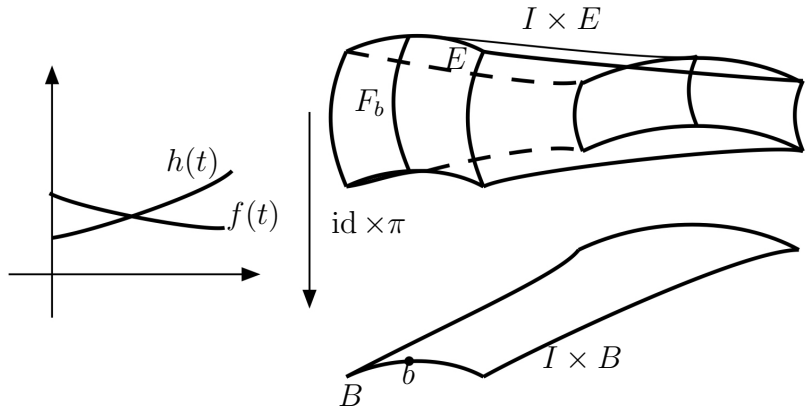
\begin{figure}
\centering 
 \begin{tikzpicture}[scale=.2]
   \node (img)  {\includegraphics[scale=0.2]{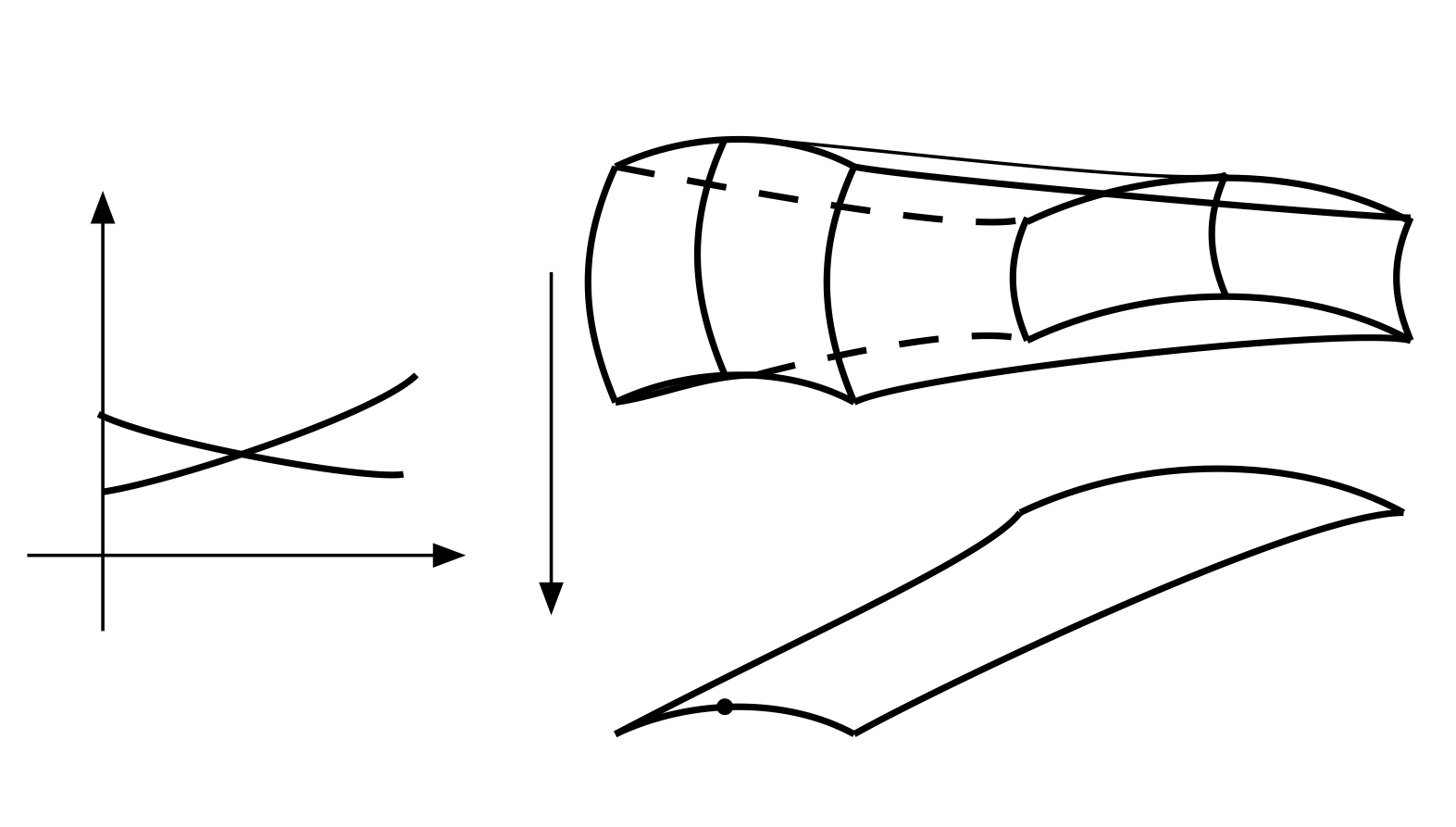}};
 \node[left] at (-12,2) {$h(t)$};
 \node[right] at (-13,-1.5) {$f(t)$};
 \node[right] at (-6,-2) {$\ID\times \pi$};
 \node[left] at (-1,6) {$F_b$};
 \node[left] at (1,-12) {$b$};
 \node[below] at (-4,-12) {$B$};
 \node[left] at (4.5,9) {$E$};
 \node[above] at (10,10) {$I\times E$};
 \node[right] at (8,-11) {$I\times B$};
  \end{tikzpicture}
  \caption{Schematic of a doubly warped Riemannian submersion metric with specified $f(t)$ and $h(t)$. }
  \label{fig:doublefig}
\end{figure}

We would now like to explain how to compute the Ricci curvatures of doubly warped Riemannian submersion metrics from the work in \cite{Wr3} on a slightly larger family of metrics. The fundamental curvature equations for any Riemannian submersions were stated and proven by O'Neill in \cite{On}, whence the equations have come to be called the O'Neill formulas. The formulas relate the curvature of total space to the curvature of the base and fiber using two tensorial obstructions: the $A$ and $T$ tensors. The $A$-tensor is the obstruction to $H$ being integrable, and the $T$-tensor is the obstruction to the $F_b^n$ being totally geodesic (see \cite[9.26]{Be}). We refer the reader to \cite[Chapter 9, Section D]{Be} for the statement of the O'Neill formulas for Riemann tensor and its traces.

One can use Proposition \ref{constructsub} to define a more general doubly warped Riemannian submersion metric $\tilde{g}$ on $M^p\times E^{n+m}$ making $\ID \times \pi:(M^p\times E^{n+m},\tilde{g})\rightarrow (M^p\times B^m, g_M+h^2(x) \check{g})$ into a Riemannian submersion for an arbitrary Riemannian manifold $(M^p,g_M)$ with fibers isometric to $(F^n,f^2(x)\hat{g})$ for any positive functions $f$ and $h$ defined on $M^p$. The class of metric considered in \cite{Wr3} is of this form with $(M^p,g_M)=(\mathbf{R}^p, dr^2 + \theta^2(r) ds_{p-1}^2)$ and functions $f$ and $h$ depending only on $r$. The $A$ and $T$ tensors of these Riemannian submersions were effectively computed, from which the Ricci curvatures are fully described in \cite[Proposition 4.2]{Wr3} in terms the functions $f$, $h$, $\theta$, and the $A$-tensor of the Riemannian submersion $\pi: (E^{n+m},g)\rightarrow (B^m,\check{g})$. Observe that our doubly warped Riemannian submersion metrics of Definition \ref{doublesubdef} are the special case $p=1$ and $\theta\equiv 0$. Below we have recorded the formulas of \cite[Proposition 4.2]{Wr3} in this special case. Notice though that we have replaced the $A$-tensor terms with Ricci curvature terms. This is achieved by simply solving for the $A$-tensor terms in the Ricci curvature O'Neill formula \cite[Proposition 9.36]{Be} applied to the Riemannian submersion $\pi:(E^{n+m},g)\rightarrow (B^m,\check{g})$ (noting that $T\equiv 0$).

\begin{lemma}\label{wraithricci}   Let $\tilde{g}$ be the doubly warped Riemannian submersion metric on $I\times E$ associated to a Riemannian submersion $\pi:(E^{n+m},g)\rightarrow (B^m,\check{g})$ with totally geodesic fibers isometric to $(F^n,\hat{g})$. The Ricci curvatures of $\tilde{g}$ are as follows
\begin{align}
\label{DRT}  \Ric_{\tilde{g}} ( \partial_t,\partial_t)  &= -m\frac{h''(t)}{h(t)} - n\frac{f''(t)}{f(t)}.\\
\label{DRH}   \Ric_{\tilde{g}} (X_i,X_i) & = \Ric_{\check{g}}(\check{Y}_i\check{Y}_i)\frac{h^2(t)-f^2(t)}{h^4(t)} -\frac{h''(t)}{h(t)} -(m-1)\frac{h'^2(t)}{h^2(t)} -n \frac{f'(t)h'(t)}{f(t)h(t)}  \\
\nonumber &\quad\qquad\qquad\qquad\qquad\qquad\qquad\qquad\qquad\qquad\qquad\qquad\qquad+ \Ric_g(Y_i,Y_i) \frac{f^2(t)}{h^4(t)}.\\
 \label{DRV}  \Ric_{\tilde{g}} (U_j,U_j) & =  \frac{\Ric_{\hat{g}}(V_j,V_j)- (n-1)f'^2(t)}{f^2(t)}- \frac{f''(t)}{f(t)}  - m \frac{f'(t)h'(t)}{f(t)h(t)} \\
\nonumber &\qquad \qquad\qquad\qquad\qquad\qquad\qquad\qquad\qquad+( \Ric_g(V_j,V_j) - \Ric_{\hat{g}}(V_j,V_j)) \frac{f^2(t)}{h^4(t)}.  \\
\nonumber \Ric_{\tilde{g}} (X_i,U_j) & = \Ric_g(Y_i,V_j)\frac{f(t)}{h^3(t)}. \\
\nonumber\Ric_{\tilde{g}}(X_i,X_j) & = \frac{1}{h^2(t)}\Ric_g(Y_i,Y_j).\\
\nonumber \Ric_{\tilde{g}}(U_i,U_j)& = \frac{1}{f^2(t)} \Ric_g (V_i,V_j).\\
\nonumber\Ric_{\tilde{g}}(X_i,\partial_t) &  =\Ric_{\tilde{g}}(U_i,\partial_t)=0.
\end{align}
\end{lemma}


\subsection{Metrics on Quotients}\label{quotient}


We are interested in applying doubly warped Riemannian submersion metrics to construct metrics on $\mathbf{P}^n\setminus D^{dn}$, which as explained in Section \ref{introduction} can be thought of as the disk bundle $D(E)$ of the normal bundle $E$ of the embedding $\mathbf{P}^{n-1}\hookrightarrow \mathbf{P}^n$. In turn, this disk bundle can be thought of as a fiber-wise quotient of $I\times S(E)$. In this section, we explain the conditions necessary for the doubly warped Riemannian submersion metric defined on  $I\times S(E)$ to descend to a metric on $D(E)$. 

To begin, we recall the situation of a trivial disk bundle $D^n\times B^m$. The doubly warped product metric $dt^2 +f^2(t)ds_{n-1}^2+ h^2(t)\check{g} $ on $[a,b]\times S^{n-1} \times B^m$ can descend to the disk bundle by allowing the warping function $f(t)$ to become zero at either $a$ or $b$. There are two main conditions that the warping functions must satisfy in order for the metric to be smooth at the cone point. The function which vanishes must be odd satisfying equation (\ref{cone}) and the second function be even satisfying equation (\ref{rot}). 
\begin{equation}\label{cone} \varphi^{(\text{even})}(a)=0\text{ and }\varphi'(a)=1,\text{ or }\varphi^{(\text{even})}(b)=0\text{ and }\varphi'(b)=-1.  \end{equation}
\begin{equation}\label{rot} \varphi(a)>0 \text{ and }\varphi^{(\text{odd})}(a)=0,\text{ or } \varphi(b)>0 \text{ and }\varphi^{(\text{odd)}}(b)=0. \end{equation}
If $f(t)$ satisfies equation (\ref{cone}) and $h(t)$ satisfies (\ref{rot}) at either $a$ or $b$, then the metric descends to $D^n\times B^m$; if they satisfy these equations at both $a$ and $b$, the metric descends to $S^n\times B^m$; and in the case $(B^m,\check{g})=(S^m,ds_m^2)$ if $f(t)$ and $h(t)$ satisfy both these equations at $a$ but interchange roles at $b$, then the metric descends to $S^{n+m+1}$ (see \cite[Proposition 1.4.7]{Pet} for more detail).  

Next, assume that we are given a rank-$(n+1)$ vector bundle $E$ over a Riemannian manifold $(B^m,\check{g})$ with bundle metric $\mu$ and connection $\nabla$. Using Proposition \ref{Vilms} this defines a metric $g$ on $S(E)$ making $\pi:(S(E),g)\rightarrow (B^m,\check{g})$ a Riemannian submersion with fibers isometric to $(S^n,ds_n^2)$. This fibration will play the role of $B^m\times S^n$ above. 

\begin{prop}\label{mainexample} Suppose we are given a Riemannian submersion $(S(E),g)\rightarrow (B^m,\check{g})$ with $O(n+1)$-invariant fiber metrics and distribution $H$. Suppose $f$ and $h$ are nonnegative functions defined on $[a,b]$ that are positive on $(a,b]$, so that $f$ satisfies (\ref{cone}) and $h$ satisfies (\ref{rot}) at $a$, then the doubly warped Riemannian submersion metric $\tilde{g}= dt^2 +h^2(t)\mathcal{H}^*\pi^*\check{g} + f^2(t)\mathcal{V}^* \hat{g}$ on $[a,b]\times S(E)$ associated to $g$ descends to a smooth metric on $D(E)$. 
 \end{prop}

 \begin{proof} As the fiber metric $\hat{g}$ is $O(n+1)$ invariant, it must be that $(F_b^n,\hat{g})\cong (S^n,ds_n^2)$. By formula (\ref{2warpedmetricform}) we have $$\tilde{g} = dt^2 + h^2(t) \mathcal{H}^*\pi^* \check{g} + f^2(t) \mathcal{V}^* ds_{n}^2.$$

Fix a point $b\in B$, and consider the bundle $T(I\times S(E))\cong TI\oplus TS(E)$ restricted to $I\times F_b^n\cong I\times S^n$. As $\mathcal{H}$ is invariant under $O(n+1)$ we see that this bundle is isomorphic to $TI\oplus T_bB^m\oplus TS^n$. The $2$-tensor $\tilde{g}$ restricted to this bundle is therefore isometric to the following. 
$$dt^2 + h^2(t)\check{g}_b+ f^2(t) ds_n^2,$$
where $\check{g}_b$ is the metric $\check{g}$ restricted to the point $b\in B$. By \cite[Proposition 1.4.7]{Pet}, the assumption on $f(t)$ at $t=a$ implies that $dt^2+f^2(t)ds_k^2$ defines a smooth metric on $D^{n+1}$, and for any fixed horizontal vectors $X,Y\in T_bB^m$, the function $h^2(t)\check{g}_b(X,Y)$ is even. It follows that the symmetric 2-tensor $h^2(t)\check{g}_b$ is smooth on $D^{n+1}$. Thus $\tilde{g}$ descends to a smooth $2$-tensor on the bundle $TB^m\oplus TD^{n+1}$ over $I\times D^{n+1}$. 

This argument holds for all $b$, which shows that $\tilde{g}$ descends to a smooth metric on $I\times S(E)/ \sim$ where $ \{0\}\times F_b^n \sim \{0\}\times\{b\}$. The effect of this quotient is to fiberwise close the cylinder $I\times S^n$ to $D^{n+1}$, coinciding therefore with $D(E)$ (see figure \ref{fig:diskbundles}). \end{proof}

 \begin{figure}
\centering

 \begin{tikzpicture}[scale=.08]
 \node (img)  {\includegraphics[scale=0.08]{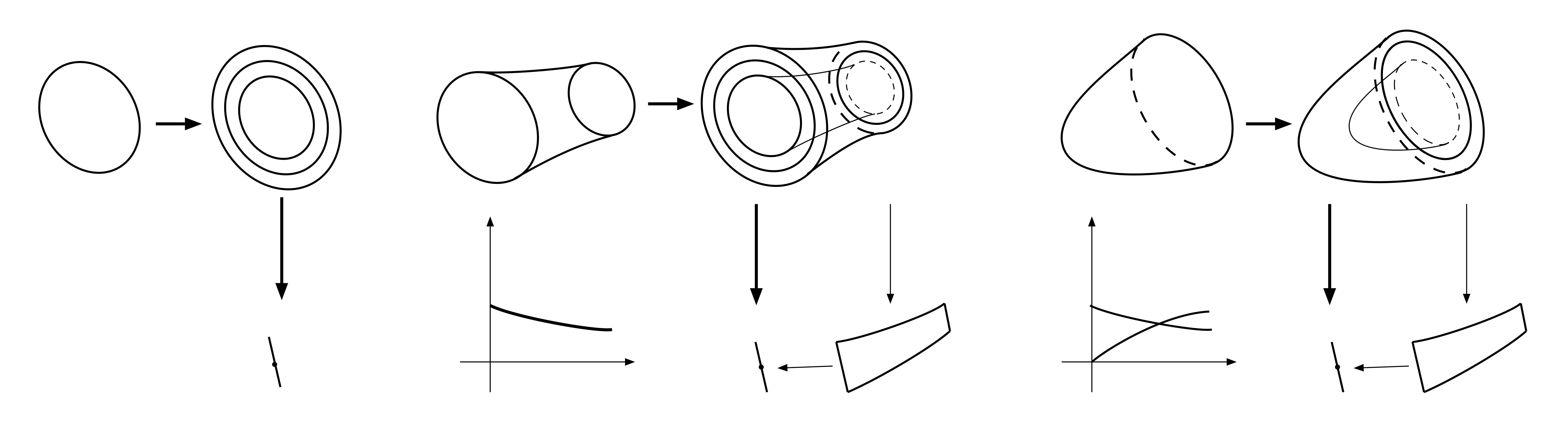}};
\node[above] at (1,-6) {$\pi$};
\node[above] at (72,-6) {$\pi$};
\node[above] at (-60,-6) {$\pi$};
\node[above] at (17,-6) {$\tilde{\pi}$};
\node[above] at (89,-6) {$\tilde{\pi}$};
\node[above] at (3,-19) {$\pi_t$};
\node[above] at (75,-19) {$\pi_t$};
\node[below] at (-24,-3.4) {$f(t)=h(t)$};
\node[above] at (68,-22) {$b$};
\node[above] at (-5,-22) {$b$};
\node[above] at (-66,-22) {$b$};
\node[below] at (70,-22) {$B^m$};
\node[below] at (-3,-22) {$B^m$};
\node[below] at (-63,-22) {$B^m$};
\node[below] at (85,-22) {$I\times B^m$};
\node[below] at (12,-22) {$I\times B^m$};
\node[below] at (58,-5) {$f(t)$};
\node[below] at (58,-11) {$h(t)$};
\node[above] at (46,25) {$F_b^n\cong D^{n+1}$};
\node[above] at (85,25) {$D(E)\cong I\times S(E)/\sim $};
\node[above] at (-29,25) {$F_b^n\cong I\times S^{n}$};
\node[above] at (5,25) {$ I\times S(E)$};
\node[above] at (-64,25) {$ S(E)$};
\node[above] at (-85,25) {$ F_b^n\cong S^{n}$};
 \end{tikzpicture}

 \caption{The Riemannian submersion $\pi:(S(E),g)\rightarrow(B^m,\check{g})$, a standard doubly warped Riemannian submersion metric with respect to this Riemannian submersion, and forming the disk bundle from a quotient of the doubly warped Riemannian submersion metric. }
  \label{fig:diskbundles}
\end{figure}
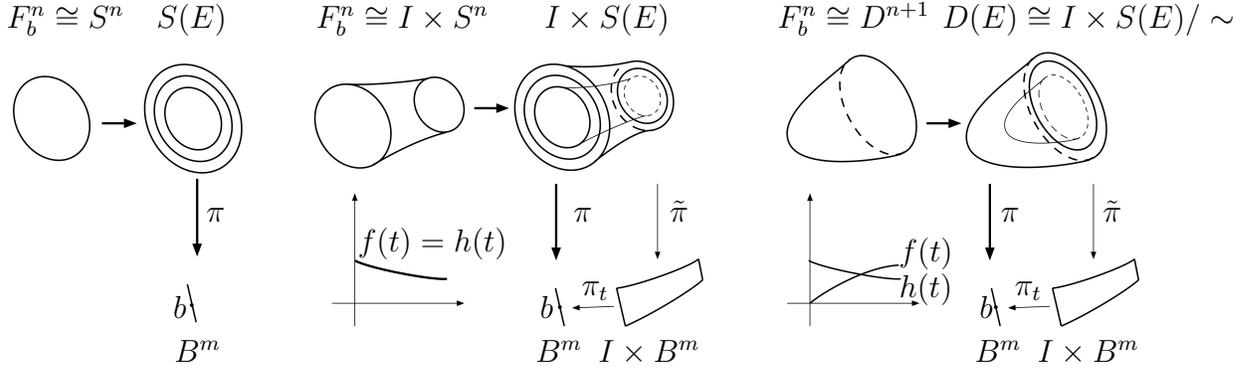

Recall that our aim is to use Proposition \ref{mainexample} to construct core metrics, in particular we need to know under what conditions will the boundary be convex. The following is a straightforward application of \cite[Proposition 3.2.1]{Pet}, it follows that the boundary will be convex if $f'(t)$ and $h'(t)$ are positive. 

\begin{lemma}\label{2warpedsecond} The second fundamental form of $(\{t\}\times E^{n+m},\tilde{g})$ as a submanifold of $([0,t]\times E^{n+m},\tilde{g})$ with respect to the normal $\partial_t$ is given by
\begin{equation} \label{2warpeq} \2_t = \frac{h'(t)}{h(t)} (h^2(t)\mathcal{H}^*\pi^* \check{g}) + \frac{f'(t)}{f(t)}(f^2(t) \mathcal{V}^*\hat{g}).\end{equation}
Where the notation is as in formula (\ref{2warpedmetricform}).
\end{lemma}


\section{Construction of the Core}\label{core}


In this section, we present the proof of Theorem \ref{A}, that there exists a Ricci positive metric on $\mathbf{P}^n\setminus D^{dn}$ so that the boundary is a convex, round sphere. In Section \ref{projectivespace}, we recall that $\mathbf{P}^n\setminus D^{dn}$ can be thought of as disk bundles of the tautological bundle. Thus Proposition \ref{mainexample} allows us to use doubly warped Riemannian submersion metrics to define metrics on $\mathbf{P}^n\setminus D^{dn}$.  We prove Theorem \ref{A} in Section \ref{theproof}, by applying Lemma \ref{wraithricci} in the particular case of projective spaces and picking specific choices of warping functions. We finish in Section \ref{perelmancore} by showing how our construction agrees with the definition present in \cite{Per1}.


\subsection{The Form of the Metric}\label{projectivespace}


Recall that the tautological bundle $\gamma_{n-1}$ over $\mathbf{P}^{n-1}$ is a rank $d$ vector bundle, such that the total space of the associated sphere bundle $\pi:S(\gamma_{n-1})\rightarrow \mathbf{P}^{n-1}$ is diffeomorphic to $S^{dn-1}$ with fiber $S^{d-1}$. These sphere bundles are called the \emph{generalized Hopf fibrations}.

Typically, one uses the Hopf fibrations to define a metric on the base of the bundle. The Fubini-Study metric is the unique metric $\check{g}$ on $\mathbf{P}^{n-1}$ that makes the submersion $\pi:(S^{dn-1},ds_{dn-1}^2)\rightarrow (\mathbf{P}^{n,-1} \check{g})$ a Riemannian submersion with fibers isometric to $(S^{d-1},ds_{d-1}^2)$. As the Hopf fibrations are the unit sphere bundles of Euclidean vector bundles $\gamma_n$, Proposition \ref{Vilms} allows us to reverse this relationship and to construct new metrics on $S(\gamma_{n-1})$ by specifying a metric on the base and vector bundle data. In particular, if we use the Fubini-Study metric on the base and use bundle data for $\gamma_n$ as in Proposition \ref{Vilms} we get back $(S^{dn-1},ds_{dn-1}^2)$. 

\begin{corollary}\cite[Example 1.4.12]{Pet}\label{fubini} Let $\check{g}$ denote the Fubini-Study metric on $\mathbf{P}^{n-1}$, and let $H$ and $V$ be the horizontal and vertical distributions determined by $\gamma_n$ as in Proposition \ref{Vilms}. Then 
$$ds_{dn-1}^2 = \mathcal{H}^*\pi^*\check{g} + \mathcal{V}^* ds_{d-1}^2.$$
\end{corollary}

That the construction in Proposition \ref{Vilms} is then compatible with Proposition \ref{mainexample}, allows us to define doubly warped Riemannian submersion metrics on $D(\gamma_{n-1})$ using the Fubini-Study metric, the bundle data coming from $\gamma_n$, and two functions $f(t)$ and $h(t)$ satisfying the hypotheses of Proposition \ref{mainexample}. As $\mathbf{P}^n\setminus D^{dn}\cong D(\gamma_{n-1})$, these doubly warped Riemannian submersion metrics can be considered as metrics on $\mathbf{P}^n\setminus D^{dn}$. Henceforth, we will refer to doubly warped Riemannian submersion metrics on $\mathbf{P}^n\setminus D^{dn}$ without confusion. Note that Corollary \ref{fubini} along with formula (\ref{2warpedmetricform}) implies that the metric restricted to level sets $t=c$ will be round if and only if $f(c)=h(c)$. 

The Hopf fibrations are very special in that they are Riemannian submersions for which the base, fiber, and total space are all endowed with Einstein metrics. This greatly simplifies the formulas of Lemma \ref{wraithricci}, as all of the Ricci tensors on the righthand side can be replaced with constants. For $(\mathbf{P}^n, \check{g})$ the constants are as follows. 

\begin{prop}\cite[Examples 9.81, 9.82, and 9.84]{Be} \label{einstein}  If $\check{g}$ is the Fubini-Study metric on $\mathbf{P}^n$ over an algebra of real dimension $d$, then
$$\Ric_{\check{g}} = \left[(n-1)d + 4(d-1)\right]\check{g}.$$
\end{prop}


\subsection{Ricci Curvature of the Cores}\label{theproof}


We are now ready to construct core metrics for $\mathbf{P}^n$ proving Theorem \ref{A}.

\begin{proof}[Proof of Theorem \ref{A}] By Proposition \ref{mainexample}, we can specify a doubly warped Riemannian submersion metric $\tilde{g}$ on $\mathbf{P}^n\setminus D^{nd}$ using the Fubini-Study metric $\check{g}$ on the base, bundle data coming from $\gamma_n$, and by specifying any functions $f(t)$ and $h(t)$ on $[0,t_1]$ that satisfy the hypotheses of Proposition \ref{mainexample}. We claim for a certain choice of $f(t)$ and $h(t)$ that $\tilde{g}$ will be a core metric, i.e. that $\Ric_{\tilde{g}}$ is positive definite and the boundary is round and convex. 

In this situation, $g=ds_{dn-1}^2$, $\hat{g}=ds_{d-1}^2$, and $\check{g}$ is the Fubini-Study metric. As observed, the underlining Riemannian submersions $\pi:(S^{dn-1},g)\rightarrow(\mathbf{P}^{n-1},\check{g})$ have totally geodesic fibers, and the metrics of the total space, base, and fiber are all Einstein. The Einstein constants of the Fubini-Study metrics are recorded in Proposition \ref{einstein}, and the Einstein constant for $(S^k,ds_k^2)$ is $(k-1)$. In this situation we have the following values. 

\begin{equation} \label{DC} \dim E = dn-1, \text{  }\dim B= d(n-1), \text{ and } \dim F= d-1.\end{equation}
\begin{equation}\label{EC} \Ric_g = (dn-2 )g, \text{  } \Ric_{\check{g}}=[(n-2)d + 4(d-1)]\check{g}, \text{ and }\Ric_{\hat{g}} = (d-2)\hat{g}\end{equation}

Regardless of the choice of $f(t)$ and $h(t)$, define $t_1$ as the smallest value of $t$ for which $f(t_1)=h(t_1)$, assuming such a value exists. By Corollary \ref{fubini}, at $t=t_1$ 
$$\tilde{g} = h^2(t_1)\left( \mathcal{H}^* \pi^*\check{g} + \mathcal{V}^*ds_{d-1}^2\right) = h^2(t_1) ds_{dn-1}^2.$$
Thus the boundary is round regardless of our choice of $h(t)$ and $f(t)$. 

Consider the choice $f(t)=\sin(t)$ and $h(t)=\varepsilon<1$. This choice satisfies the hypotheses of Proposition \ref{mainexample}, and therefore define metrics on $\mathbf{P}^n\setminus D^{nd}$. Substituting this choice of $f(t)$ and $h(t)$ into the formulas in Lemma \ref{wraithricci}, and replacing the Ricci curvatures and dimensional constants with those from equations (\ref{EC}) and (\ref{DC}) yields the following. 
\begin{align*}
\Ric_{\tilde{g}}(\partial_t,\partial_t) &  =  (d-1);\\
\Ric_{\tilde{g}}(X_i,X_i)  & =  [(n-2)d + 4(d-1)]\frac{\varepsilon^2-\sin^2 t}{\varepsilon^4} + (dn-2) \frac{\sin^2 t}{\varepsilon^4};\\
\Ric_{\tilde{g}}(U_j,U_j) & = (d-2)\tan^2 t +1  +[(nd-2)- (d-2) ]\frac{\sin^2t}{\varepsilon^4}.
\end{align*}
As $\Ric_g =(nd-2) g$ in this case, the off-diagonals of $\Ric_{\tilde{g}}$ in Lemma \ref{wraithricci} all vanish. Since $0\le t \le t_1< {\pi}/2$, and $\sin(t)\le \varepsilon$, it is easy to check that the remaining Ricci curvatures are all strictly positive. Thus, for this choice of $f(t)$ and $h(t)$, $\Ric_{\tilde{g}}$ is positive definite. 

To show that the boundary is convex, by Lemma \ref{2warpedsecond} we need $h'(t_1)>0$ and $f'(t_1)>0$. For $h(t)=\varepsilon$ this is not the case. We may replace $h(t)$ with a function that has $h'(t_1)>0$, where $t_1$ is still the smallest value for which $h(t_1)=\sin(t_1)$, and such that $||h(t)-\varepsilon||_{C^2}$ is small enough so as not to upset $\Ric_{\tilde{g}}>0$. With this function $h(t)$, the claim follows. \end{proof}

In Section \ref{perelmancore} we explain that \cite{Per1} Perelman chose $h(t) = \left({1}/{N}\right)\cosh\left({t}/{N}\right)$ and $N=100$. For each $n$, there exists a choice of $N$, for which this $h(t)$ and $f(t)$ define a metric $\tilde{g}$ which satisfies the claims of Theorem \ref{A}.

We have, thus far, omitted $\RP^n$ from discussion. Since $\RP^n$ also admits a tautological bundle, one may equally well define a doubly warped Riemannian submersion metrics on $\RP^n$ as in Section \ref{projectivespace}. But because the fibers of the real Hopf fibration are $S^0$, the metric reduces to a warped product metric. In particular, instead of equation (\ref{DRT}), such a metric has 
$$\Ric_{\tilde{g}}(\partial_t,\partial_t) = -(n-1)\dfrac{h''(t)}{h(t)}.$$
 Thus $h(t)$ must be concave down for Ricci curvature to be positive. But in order for $\tilde{g}$ to be smooth, $h'(0)=0$ by Proposition \ref{mainexample}, and in order for the boundary to be convex $h'(t_1)>0$ by Lemma \ref{2warpedsecond}. There is no smooth function with $h'(0)=0$, $h''(t)<0$, and $h'(t_1)>0$ for $0<t_1$. Thus it is not possible to prove Theorem \ref{A} for $\RP^n$ using doubly warped Riemannian submersion metrics. We will see in Section \ref{lens} that there does not exist any core metrics for $\RP^n$.


\subsection{Perelman's Core}\label{perelmancore}


In Section \ref{introduction}, we explained how doubly warped Riemannian submersion metrics on $\mathbf{P}^n\setminus D^{dn}$ are an obvious generalization of Perelman's core metrics on $\CP^2\setminus D^4$ of \cite[Section 2]{Per1}. In this section, we show that doubly warped Riemannian submersion metrics of Section \ref{projectivespace} on $\CP^2\setminus D^4$ actually agree with Perelman's core metrics. Recall that $\CP^2\setminus D^4$ can be realized as a fiberwise quotient of $I\times S^3$. As $S^3$ is a Lie group, it therefore admits a globally defined left-invariant, orthonormal frame $X$, $Y$, and $Z$. We may assume that the vector field $Z$ is the image of the vector field $\Theta$ of $S^1$ under the differential of the inclusion of each fiber. Denote the dual covector fields of $X$, $Y$, and $Z$ with respect to $ds_3^2$ by $dx$, $dy$, and $dz$ respectively. After checking that $dz^2 = \mathcal{V}^* \hat{g}$ and $dx^2+dy^2 =\mathcal{H}^*\pi^*\check{g}$, formula (\ref{2warpedmetricform}) becomes $dt^2 + h^2(t)(dx^2+dy^2)+ f^2(t)dz^2.$ This is the form in which Perelman provided the metric in \cite[Section 2]{Per1}, thus our metrics agree with those considered in \cite{Per1}. To prove Theorem \ref{A} for $\CP^2$, Perelman chose $f(t)= \sin t \cos t$ and $h(t) = \left({1}/{100}\right)\cosh\left( {t}/{100} \right)$. Note that these particular choices satisfy the requirements of Proposition \ref{mainexample}. 

Perelman used the Lie group structure of $S^3$ to compute the Ricci curvature of his metrics. As there are no Lie groups diffeomorphic to $S^{dn-1}$ beyond $S^3$, this necessitates our use of \cite[Proposition 4.2]{Wr3} to compute the Ricci tensor of a doubly warped Riemannian submersion metric. The exact Ricci curvatures computed in \cite{Per2} are stated as follows. 

\begin{lemma}\cite[Section 2]{Per1} \label{perelmanricci} Let $X$, $Y$, and $Z$ be the global left-invariant vector fields of $S^3$. If $\tilde{g}= dt^2+ h^2(t)(dx^2 +dy^2)+f^2(t) dz^2$, then 
\begin{align*}
\Ric_{\tilde{g}}(\partial_t,\partial_t) & = -\frac{h''(t)}{h} - 2\frac{f''}{f},\\
\Ric_{\tilde{g}}(X,X) = \Ric_{\tilde{g}}(Y,Y) & = 4 \frac{h^2-f^2}{h^4} -\frac{h''}{h} -2\frac{h'^2}{h^2} - \frac{f'h'}{fh} + 2 \frac{f^2}{h^4},\\
\Ric_{\tilde{g}}(Z,Z) & =  -\frac{f''}{f}  - 2 \frac{f'h'}{fh} +2\frac{f^2}{h^4}, 
\end{align*}
and all other terms in this frame vanish. 
\end{lemma}

We can check that the formulas in Lemma \ref{perelmanricci} agree with those in Lemma \ref{wraithricci}.  Indeed, in the case of $\CP^2\setminus D^4$ we have $n=1$, $m=2$, $\Ric_g=ds_3^2$, $\Ric_{\hat{g}} = 0$, and $\Ric_{\check{g}}=\left({1}/{2}\right)^2ds_2^2$. Plugging these values into the formula in Lemma \ref{wraithricci} agrees Lemma \ref{perelmanricci}.


\section{Constructions of Connected Sums with Positive Ricci Curvature}\label{constructions}


In this section we show how to assemble together the cores of Theorem \ref{A} and the various constructions of \cite{Per1} to form connected sums with positive Ricci curvature. We begin by constructing the docking station of Proposition \ref{docking} in Section \ref{dockingsection} by combining the work present in \cite{Per1}, specifically the gluing Lemma \ref{glue} and two technical constructions summarized below as Lemmas \ref{actualambient} and \ref{neck}. Next, in Section \ref{assembly} we show how the existence of the docking stations of Proposition \ref{docking} combined with the gluing Lemma \ref{glue} allow us to attach core metrics and prove Theorem \ref{parts}. In Section \ref{lens}, we show how Lemma \ref{actualambient} can also be used to define a docking station metric on lens spaces, thus allowing one to generalize Theorem \ref{parts} to take a connected sum with a finite quotient of $S^n$ in Corollary \ref{quotientmanifold}.


\subsection{Constructing the Docking Station}\label{dockingsection}


The first technical result is the construction of a positively curved metric on $S^n_k$ with principal curvatures of the boundary that are small relative to the intrinsic curvatures of the boundary. This is what Perelman called \emph{the ambient space}, and the construction takes up the entirety of the the contents of \cite[Section 3]{Per1}. 

\begin{lemma} \cite[Section 3]{Per1}\footnote{Lemma \ref{actualambient}, as stated here, is not explicitly claimed in \cite{Per1}. This claim roughly reflects the second paragraph of page 162 in \cite{Per1}.}\label{actualambient} For all $n\ge 4$, there exists an $\eta>0$ such that, for any integer $k$ and $0<r<\eta$, there is a metric $g_\text{ambient}$ on $S^n_k$ such that the metric restricted to each boundary component is isometric to $g_\delta$ such that
\end{lemma}
\begin{enumerate}[(i)]
\item $\K_{g_\text{ambient}}>0$;
\item ${\K}_{g_\delta} >1$;
\item $|\2_{g_\delta}|<1$;
\item $g_\delta = d\phi^2 +f^2(\phi)ds_{n-2}^2$ with $\phi\in[0,\pi R]$, $\sup_\phi f= r$, where $R$ is some number such that $0<r^{\frac{n-1}{n}}<R<1$. 
\end{enumerate} 

Perelman constructed this metric only with $n=4$. The metric is a doubly warped product metric, so by considering the metric $dt^2+\cos^2t dx^2 + R^2(t) ds_{n-2}^2$ where $R(t)$ is as in \cite[Section 3]{Per1}, one gets a metric on $S^n$ for all $n\ge 4$. It is easily seen that all of the curvature conditions transfer to higher dimensions. 

The core metrics constructed on punctured projective spaces in Theorem \ref{A} have round boundaries, and the ambient space metric constructed in Lemma \ref{actualambient} has boundaries isometric to warped products. Indeed any metrics constructed on $\mathbf{P}^n\setminus D^{dn}$ using Riemannian submersion metrics cannot have boundaries isometric to warped products other than the round metric. Thus if one hopes to use Lemma \ref{glue} to glue $(\mathbf{P}^n\setminus D^{dn},g_\text{core})$ onto $(S^n_k,g_\text{ambient})$, there has to be a way to transition between round and warped product metrics. Perelman resolved this with the following lemma which represents the most technical construction of \cite{Per1}, in which \emph{the neck} is constructed, which is precisely a positive Ricci metric on $I\times S^{n-1}$ that is round at one end and a warped product metric at the other with principal curvatures that allow us to glue at the round end to the cores and at the warped product end to the ambient space using Lemma \ref{glue}.

\begin{lemma}{\rm \cite[p. 159]{Per1}} \label{neck}
  Assume that $n\ge 3$, $0<r<R<1$, and $g_1=d\phi^2+ f^2_1(\phi)
  ds_{n-1}^2$ is a metric on $S^n$ with $\phi \in [0,\pi R]$,
  $\sup_\phi f_1(\phi) = r$, and $\K_{g_1}>1$. Then for any any
  $\rho>0$ satisfying $r^{(n-1)/n}<\rho<R$, there exists a metric
  $g=g(\rho)$ defined on $S^n\times [0,1]$ and constant $\lambda>0$ such
  that the following are true.
  \end{lemma}
\begin{enumerate}[(i)]
\item\label{riccicondition} $\Ric_{g}$ \emph{is positive definite;}
\item\label{0end}\emph{the restriction }$g|_{t=0}$ \emph{coincides
  with }${\rho^2}/{\lambda^2} ds_n^2$ \emph{on} $S^n\times \{0\}$;
\item\label{1end}\emph{the restriction} $g|_{t=1}$ \emph{coincides with} $g_1$
 \emph{on} $S^n\times \{1\}$;
\item\label{0extrinsic}\emph{the principal curvatures along the
  boundary} $S^n\times \{0\}$ \emph{are equal to }$-\lambda$;
\item\label{1extrinsic}\emph{the principal curvatures along the
  boundary }$S^n\times \{1\}$ \emph{are at least }$1$.
\end{enumerate}

\noindent  The metric constructed in \cite[Section 2]{Per1} to satisfy Lemma \ref{neck} is of the form $dt^2+g_t$ where $g_t$ is a one parameter family of warped product metrics. The proof of this fact very delicate, and takes up roughly half of \cite{Per1}. 

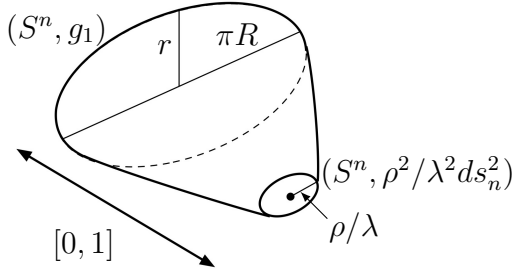
\begin{figure}
\centering
\begin{tikzpicture}[scale=.12]
\node (img) {\includegraphics[scale=0.12]{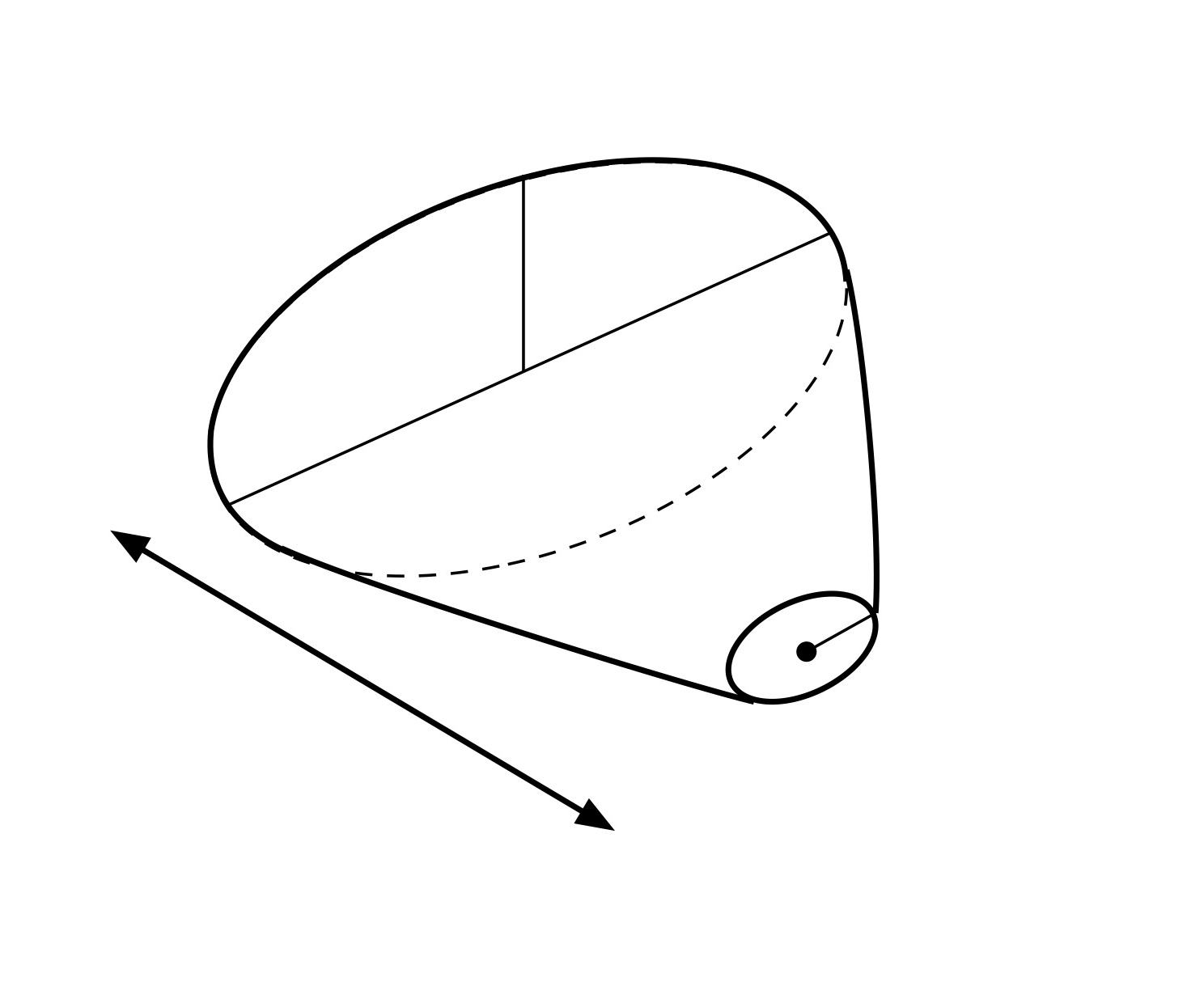}};
\node[above] at (-5,8) {$r$};
\node[above] at (3,9) {$\pi R$};
\node[above] at(-17,9) {$(S^n,g_1)$};
\node[below] at (-14,-9) {$[0,1]$};
\node[right] at (12,-10) {${\rho}/{\lambda}$};
\draw[->] (13,-9) -- (10,-6);
\node[above] at (23,-7) {$\left(S^n, {\rho^2}/{\lambda^2} ds_n^2\right)$};
\end{tikzpicture}

\caption{The neck.}
\label{fig:theneck}
\end{figure}

With all of the necessary technical constructions of \cite{Per1} summarized, we may now explain how they come together to construct the docking station of Proposition \ref{docking}.

\begin{proof}[Proof of Proposition \ref{docking}] As illustrated in Figure \ref{fig:docking}, we attach $k$ copies of $[0,1]\times S^{n-1}$ with the neck metric to $S^n_k$ with the ambient space metric using Lemma \ref{glue}. 

Fix $n>3$, $k>0$, and $\rho <1$. Pick $r<R<1$, such that $r^{\frac{n-1}{n}}<\rho < R$. By Lemma \ref{actualambient} there exists a metric $g_\text{ambient}$ on $S^n_k$ with positive Ricci curvature such that each boundary component is isometric to $(S^{n-1},g_\delta)$, where $\K_{g_\delta} >1$, $|\2_{g_{\delta}}|<1$, and $g_\delta = d\phi^2 + f^2(\phi)ds_{n-2}^2$ with $\phi\in [0,\pi R]$ and $\sup_\phi f(\phi) = r$. 

It follows that we may use $g_1=g_\delta$ as the initial data in Lemma \ref{neck}. Thus there is a metric $g_\text{neck}(\rho)$ on $[0,1]\times S^{n-1}$ with positive Ricci curvature such that the boundary at $0$ is isometric to $\left(S^{n-1},{\rho^2}/{\lambda^{2}}ds_{n-1}^2\right)$ with principal curvatures all identically $-\lambda$, and the boundary at $1$ is isometric to $\left(S^{n-1},g_\delta\right)$ and $\2_1>1$. 

Thus there is an isometry $\phi$ between the boundary of $([0,1]\times S^{n-1}, g_\text{neck})$ at $1$ and any one of the boundary components of $(S^n_k,g_\text{ambient})$. Both manifolds have positive Ricci curvature, and by construction $\phi^*\2_1 + \2_{g_\delta} > 0.$ It follows from Lemma \ref{glue}, that there exists a metric $g$ on $S^{n-1} \times [0,1] \cup_\phi S^n_k$ that agrees with $g_\text{neck}$ and $g_\text{ambient}$ away from a small neighborhood of the image $\phi$. In particular, the boundary at $0$ remains isometric $\left(S^{n-1},{\rho^2}/{\lambda^{2}}ds_{n-1}^2\right)$ and the remaining disjoint $(k-1)$ boundary components remain isometric to $(S^{n-1},g_\delta)$. Thus we may glue $k$ disjoint copies of $([0,1]\times S^{n-1}, g_\text{neck})$ to each boundary component of $(S_k^n, g_\text{ambient})$. 

\begin{figure}
\centering

\begin{tikzpicture}[scale=.15]
\node (img)  {\includegraphics[scale=0.15]{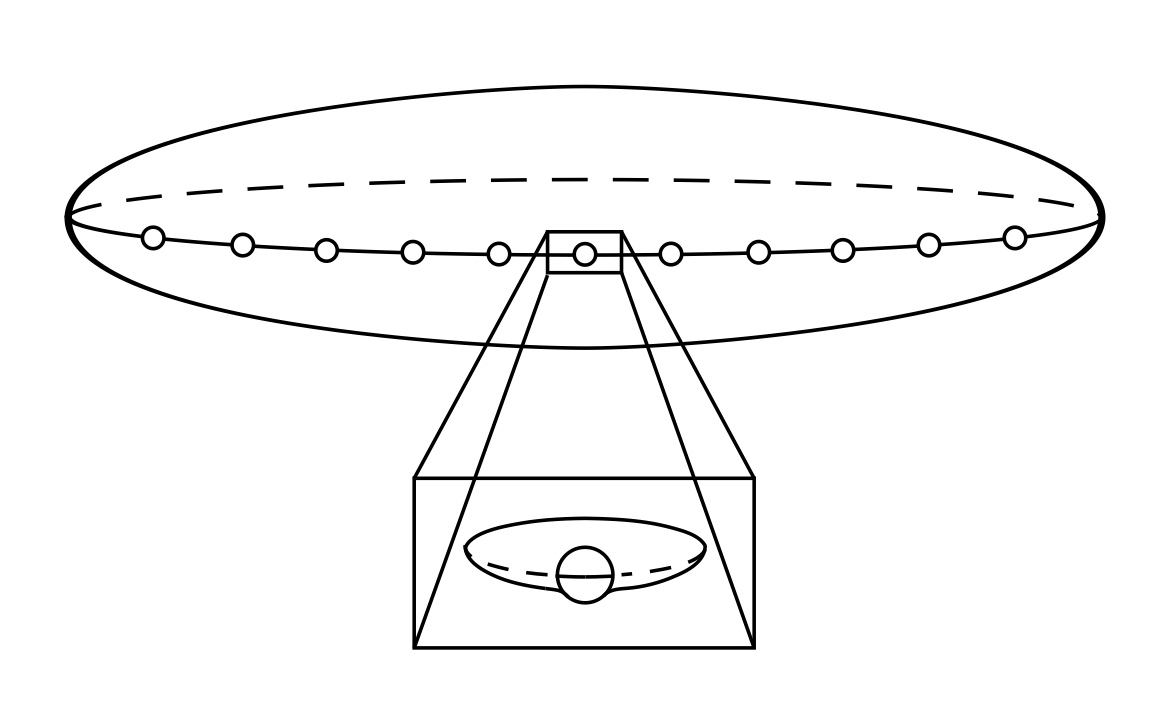}};
\small
\draw[<-] (4.2,-7) -- (8,-7);
\node[right] at (8, -6.8) {a neck};
\node[above] at (0,9.5) {the ambient space};
\end{tikzpicture}

\caption{The construction of the docking station.}
\label{fig:docking}
\end{figure}
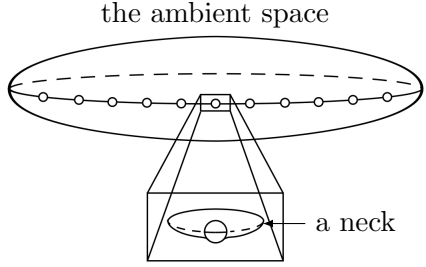

The resulting metric $g$ is defined on $S^n_k$. It has positive Ricci curvature and boundary components all isometric to $\left(S^{n-1},{\rho^2}/{\lambda^{2}}ds_{n-1}^2\right)$ with principal curvatures $-\lambda$. Set $g_\text{docking}=\lambda^2 g$. The positivity of the Ricci curvature is unaffected by scaling, each boundary component of the boundary of $(S^n_k,g_\text{docking})$ is isometric to $(S^{n-1},\rho^2ds_{n-1}^2)$, and after scaling the principal curvatures of the boundary are all $-1$ (see \cite[Theorem 1.159]{Be}).
\end{proof}

As mentioned, the docking station is the most technical aspect of \cite{Per1} as it involves proving Lemma \ref{neck}. One may ask if this construction is really necessary for the proof of Theorem \ref{parts} in Section \ref{assembly}. Suppose we tried to use the round $n$-sphere as a docking station. In order to use Lemma \ref{glue} to attach $M^n\setminus D^n$, if its round boundary has radius $r$, then we must assume that the principal curvatures of the boundary are at least $\cot r$. It is a consequence of \cite[Theorem 1]{Wa2}, that there is an $r_n>0$ such that $M$ is contractible if $r<r_n$.\footnote{In this situation one can show that the convexity invariant of \cite{Wa2} satisfies $\Lambda(M\setminus D^n)\ge \cos r$ and that the curvature satisfies $K>\csc^2 r+\cot^2 r$ at the boundary. Thus the hypotheses of \cite[Theorem 1]{Wa2} are met for small enough $r$} So the round sphere cannot be used to construct Ricci positive metrics in this way for arbitrary connected sums. In some sense, the construction of ambient space was the simplest possible way to achieve this goal, which in turn necessitated the construction of the neck.


\subsection{Attaching the Cores}\label{assembly}


In this section, we explain how to attach cores to the docking station of Proposition \ref{docking} using the gluing Lemma \ref{glue} as illustrated in Figure \ref{fig:connected}, thus proving Theorem \ref{parts}. 

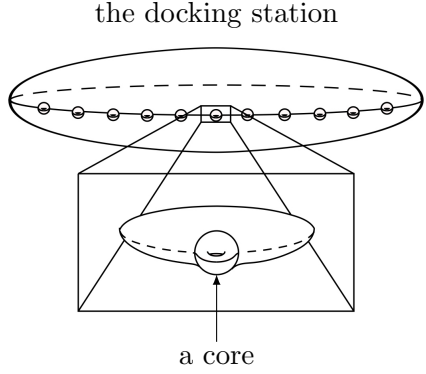
\begin{figure}
\centering

\begin{tikzpicture}[scale=.15]
\node (img) {\includegraphics[scale=0.15]{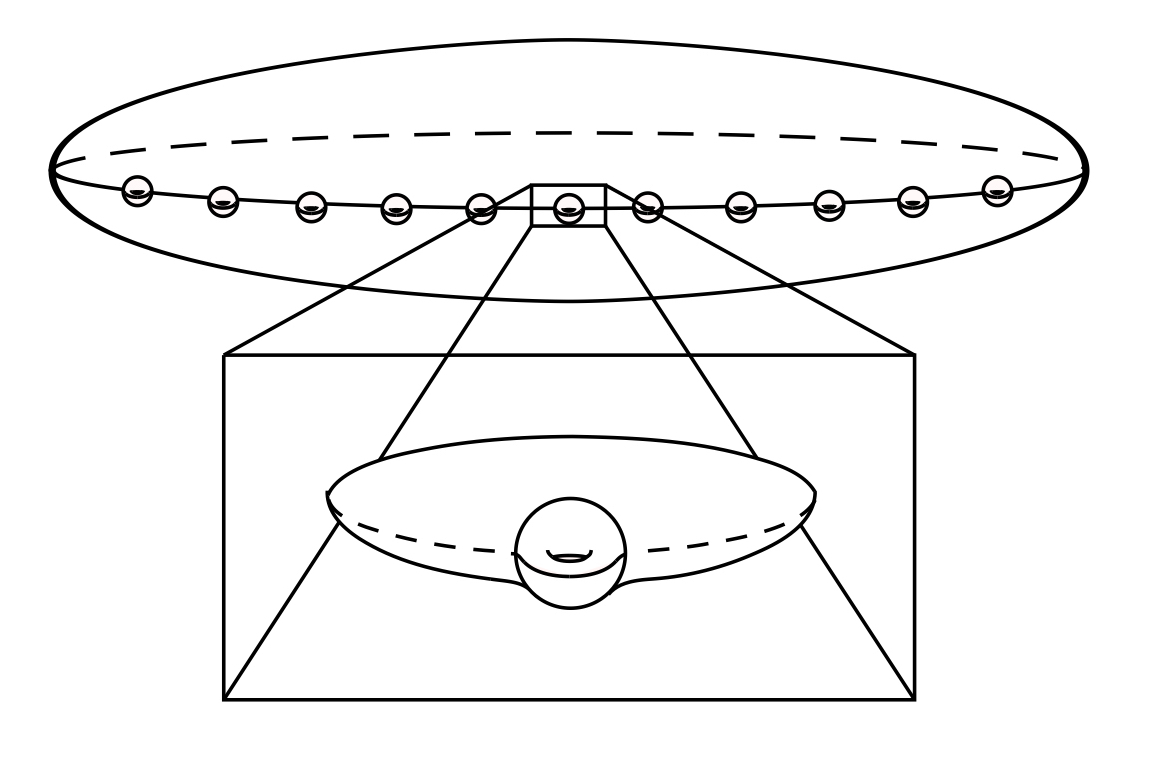}};
\draw[<-] (-.3,-8) -- (-.3,-14);
\small
\node[below] at (-.3,-14) {a core};
\node[above] at (-.3,13) {the docking station};
\end{tikzpicture}
\caption{The construction $\#_{i=1}^k M_i^n$ from the docking station.}
\label{fig:connected}
\end{figure}

\begin{proof}[Proof of Theorem \ref{parts}] 

Suppose we want to take a connect sum of manifolds $M_i$ with $1\le i\le k$. Let $g_i$ be the metrics on $M_i\setminus D^n$ with positive Ricci curvature and round, convex boundary $N_i$. Let $\rho_i$ be the radius of $N_i$ and let $\nu_i = \inf_{v\in S(TN_i)} \2_{g_i}(v,v)$. Pick a small number $0<\rho <1$ and define $s_i$ to be the number such that $s_i\rho_i= \rho$. Assume that $\rho$ is so small that ${\nu_i}/{s_i} >1$. The manifolds $(M_i\setminus D^n, s_i^2 g_i)$ have positive Ricci curvature, boundaries isometric to $(S^{n-1},\rho^2 ds_{n-1}^2)$, and the principal curvatures of the boundaries all greater than ${\nu_i}/{s_i}>1$ (see \cite[Theorem 1.159]{Be}).

By Proposition \ref{docking}, there exists a Ricci positive metric $g_\text{docking}$ on $S^n_k$ with boundaries isometric to $(S^{n-1},\rho^2 ds_{n-1}^2)$ and principal curvatures all equal $-1$. We clearly have isometries $\phi_i$ between $N_i$ and the boundary components of $(S^n_k,g_\text{docking})$. By Lemma \ref{glue}, we can glue each of the $(M_i^n\setminus D^n,s_i^2 g_i)$ to $(S^n_k,g_\text{docking})$ along $\phi_i$ so that the resulting space admits a Ricci positive metric. The resulting space is clearly diffeomorphic to $\#_{i=1}^k M_i^n$. 
\end{proof}

With this, we have proven that the connected sums of all complex, quaternionic, and octonionic projective spaces admit metrics with positive Ricci curvature.

\begin{proof}[Proof of Theorem \ref{B}] By Theorem \ref{A}, for all $n>1$, $\CP^n$, $\HP^n$, and $\OP^2$ admit core metrics. Theorem \ref{B} now follows directly from Theorem \ref{parts}. 
\end{proof}

\noindent Notice that one may just as well take connected sums between the different projective spaces in their common dimensions. Thus the following is immediate. 

\begin{corollary}\label{justproj} For all $n\ge 1$ and $j,k,l\ge 0$, the following manifolds admit metrics with positive Ricci curvature.
\begin{enumerate}
\item $\left(\#_j\CP^{2n}\right) \#\left( \#_k \HP^n\right)$.
\item $\left(\#_j\CP^{8}\right)\#\left(\#_k \HP^{4}\right)\#\left(\#_l\OP^2\right)$.
\end{enumerate}
\end{corollary}


\subsection{Lens Spaces as Docking Stations}\label{lens}


 In this section, let $G$ denote a finite subgroup of the isometry group of $(S^n,g)$. If $G$ acts freely, then $(S^n/G,g)$ is a smooth manifold locally isometric to $(S^n,g)$ (\cite[Theorem 21.13]{Lee2}). In particular, the metric $g_\text{ambient}$ of Lemma \ref{actualambient} is a doubly warped product metric with isometry group $O(2)\oplus O(n-1)$. One might hope then that some nontrivial quotients of $S^n$ might play the role of docking station in Proposition \ref{docking}, allowing us to form new examples of connected sums with positive Ricci curvature. Denote by $(S^n/G)_k$ the quotient manifold with $k$ disjoint geodesic balls removed. 

\begin{corollary}\label{lensdocking}  Suppose that $G$ is a finite subgroup of $O(2)\oplus O(n-1)\le O(n+1)$ and that the action of $O(n+1)$ on $S^n$ restricted to $G$ is free. For any $n>3$, $k>0$, and $0<\rho<1$, there is a metric $g_\text{lens}$ on $(S^n/G)_k$ with positive Ricci curvature and so that each boundary component is isometric to $(S^{n-1},\rho^2 ds_{n-1}^2)$ with all principal curvatures equal to $-1$. 
\end{corollary}

\begin{proof} Start by choosing $0<r<R<1$ such that $r^{\frac{n-1}{n}}<\rho<R$. Use Lemma \ref{actualambient} to define a Ricci positive metric $g_\text{ambient}$ on $S^n_l$ with $l=|G|k$ with boundaries all isometric to $(S^{n-1},g_\delta)$. As the action of $G$ is free, it is possible to choose the boundary $S^{n-1}$ disjoint and invariant under the action of $G$. It follows that the quotient $(S^n/G)_k$ has a Ricci positive metric with $k$ disjoint boundary components all isometric to $(S^{n-1},g_\delta)$ and principal curvatures equal to $-1$. The proof now precedes identically to the proof of Proposition \ref{docking} in Section \ref{dockingsection}
\end{proof}

Thus $(S^n/G)_k$ can now replace $S^n_k$ in the construction in the proof of Theorem \ref{parts} in Section \ref{assembly}. The effect topologically is taking a connected sum with $S^n/G$. Thus the following is immediate. 

\begin{corollary}\label{quotientmanifold} For all $k\ge 1$ and $n\ge 4$, suppose that $G$ is a finite subgroup of $O(2)\oplus O(n-1)\le O(n+1)$ and that the action of $O(n+1)$ on $S^n$ restricted to $G$ is free.  If there exists Ricci positive metrics $g_i$ on $M^n_i\setminus D^n$ with round, convex boundaries, then the following manifold admits a Ricci positive metric.
$$(S^n/G)\#\left(\#_{i=1}^k M^n_i\right).$$
\end{corollary}

One may consult the literature on finite subgroups $G\le O(n+1)$ and ask which groups are conjugate to subgroups of $O(2)\oplus O(n-1)$. The simplest example of such subgroups is when $G=\mathbf{Z}/m\mathbf{Z}$ and quotients $S^n/G$ are lens spaces. 

\begin{definition}\cite[Example 2.43]{Ha}\label{lensspace}
For a fixed positive integer $m$, for $1\le j\le n$, let $\ell_j$ be integers with $\gcd(m,\ell_j)=1$. Let $\theta_j=2\pi\ell_j/m$. Define an action of $\mathbf{Z}/m\mathbf{Z}$ on $S^{2n-1}\subseteq \mathbf{C}^n$ by the standard action of $\text{diag}( \exp(i\theta_1),\dots,\exp(i\theta_n))$ on $\mathbf{C}^n$. Then define the lens space associated to the tuple $(m,\ell_1,\dots, \ell_n)$ by the quotient of this action
$$L(m; \ell_1,\dots,\ell_n) = S^{2n-1}/(\mathbf{Z}/m\mathbf{Z}).$$
\end{definition}

In particular, $L(2;1,\dots, 1)\cong \RP^{2n-1}$. Notice that under an $\mathbf{R}$-linear isomorphism $\mathbf{C}^n\cong \mathbf{R}^{2n}$, the action of $\mathbf{Z}/m\mathbf{Z}$ defined above acts as a subgroup of $ O(2)\oplus O(2n-2)$. Thus all lens spaces occur as possibilities of $S^{2n-1}/G$ in Corollary \ref{quotientmanifold}. While $\RP^{2n}$ is not a lens space, we see that $-I\in O(2)\oplus O(n-2)$ for all $n$ and therefore $\RP^{2n}$ will also occur as $S^{2n}/G$. The following is immediate. 

\begin{corollary}\label{lensconnect} For all $k\ge 2$, $n\ge 3$, and $d\ge4$, if there are Ricci positive metrics on $M_i^n\setminus D^n$ with round, convex boundaries, then the following manifolds admit metrics with positive Ricci curvature
$$L(m; \ell_1,\dots, \ell_n )\# \left( \#_{i=1}^k M_i^{2n-1} \right) \text{ and } \RP^d\# \left(\#_{i=1}^k M_i ^d\right). $$
\end{corollary}


As it stands, we have not constructed, nor are we aware of any odd dimensional manifolds admitting core metrics other than $S^n$. Thus we do not have an application to the statement involving $L(m;\ell_1,\dots, \ell_n)$ in Corollary \ref{lensconnect}. We therefore pose the following question. 
 
\begin{Q} Does there exist an odd dimensional manifold $M^{2n+1}$ that admits a Ricci positive metric on $M^{2n+1}\setminus D^{2n+1}$ with round, convex boundary? 
\end{Q} 

In even dimensions, Theorem \ref{A} provides the existence of core metrics on complex, quaternionic, and octonionic projective spaces. Thus Corollary \ref{lensconnect} provides new examples in the case of $\RP^{2n}$. In particular this proves the following corollary.

\begin{corollary}{$\text{A}^\prime$}\label{realprojective} For all $n\ge 1$, and $j,k,l \ge 0$ the following manifolds admit metrics with positive Ricci curvature.
\end{corollary}
\begin{enumerate}[(i)]
\item $\RP^{2n}\#\left(\#_j\CP^n\right)$.
\item $\RP^{4n}\#\left(\#_j\CP^{2n}\right) \#\left( \#_k \HP^n\right)$.
\item $\RP^{16}\#\left(\#_j\CP^{8}\right)\#\left(\#_k \HP^{4}\right)\#\left(\#_l\OP^2\right)$.
\end{enumerate}

While lens spaces may play the role of the docking station, if they admit core metrics, by Corollary \ref{lensconnect} we would have proven that $L(m;\ell_1,\dots, \ell_n)\# L(m'; \ell_1',\dots, \ell_n')$ admits a Ricci positive metric.  Myers' theorem (see \cite[Theorem 6.3.3]{Pet}) implies that a closed manifold with positive Ricci curvature has finite fundamental group. If $M^n$ and $N^n$ with $n\ge 3$ are manifolds with nontrivial fundamental groups, then $\pi_1(M\# N)$ is infinite by the Seifert-van Kampen theorem. As $\pi_1(L(m;\ell_1,\dots, \ell_n))=\mathbf{Z}/m\mathbf{Z}$, $L(m;\ell_1,\dots, \ell_n)\# L(m'; \ell_1',\dots, \ell_n')$ must have an infinite fundamental group, and therefore cannot admit a metric with positive Ricci curvature. We conclude that lens space cannot admit core metrics, indeed no space with nontrivial fundamental group can. 

\begin{corollary} For $n\ge 4$, if $M^n$ is not simply connected, then there are no Ricci positive metrics on $M^n\setminus D^n$ with round, convex boundaries. In particular, $\RP^n$ and $L(m,\ell_1,\dots, \ell_k)$ do not admit such metrics. 
\end{corollary}

\noindent While we have explained in our remark after the proof of Theorem \ref{A} that the methods of this paper do not work constructing core metrics on $\RP^n$, the above corollary shows that no such metric exists.


\section{Connected Sums of Product Spaces}\label{productconnect}


In this section, we show how to take the metric constructions of \cite{Per1} and Section \ref{core} and combine them with topological observation of \cite{Sha} to construct Ricci positive metrics on connected sums of products. Specifically we will prove the following. 

\begin{corollary}\label{productproj} For all $n\ge 1$, $m\ge 3$, and $i,j,k,l\ge 0$ the following manifolds admit metrics of positive Ricci curvature. 
\begin{enumerate}
\item $(\#_i (S^{2n}\times S^m)) \#( \#_j (\CP^n\times S^m)) $.
\item $(\#_i (S^{4n}\times S^m)) \# (\#_j (\CP^{2n}\times S^m)) \# ( \#_k (\HP^n\times S^m))$.
\item $(\#_i (S^{16}\times S^m)) \#(\#_j (\CP^8 \times S^m)) \# (\#_k( \HP^4 \times S^m)) \# (\#_l (\OP^2 \times S^m))$.
\end{enumerate}
\end{corollary}

\noindent As explained in Section \ref{introduction} this is a generalization of Theorem \ref{sharesult}. The approach in \cite{Sha} to proving Theorem \ref{sharesult}, was to prove a surgery theorem for Ricci positive manifolds (see \cite[Lemma 1]{Sha} or \cite[Theorem 0.3]{Wr2}) under the assumption of an isometrically embedded standard $S^{n-1}\times D^{m+1}$. The key topological observation that yields Theorem \ref{sharesult} as a corollary is that 
\begin{equation} \label{shasurgery} \#_k( S^n\times S^m) \cong \left( S^{n-1} \times \left( S^{m+1} \setminus \bigsqcup_{i=0}^k D^{m+1}\right)\right)\cup_\partial \left(\bigsqcup_{i=0}^k D^n\times S^m \right).\end{equation}

Our approach to proving Corollary \ref{productproj} will rely on an observation like equation (\ref{shasurgery}), namely Proposition \ref{decomp}. Indeed, we will show that one may replace $S^n$ with $M_i^n$ on the lefthand side if one replaces $k$ of the $D^n$ on the righthand side with $M^n_i\setminus D^n$. That said, we do not have a surgery theorem for positive Ricci curvature present in this paper, instead we rely on docking station metrics of Proposition \ref{docking}, which will allow us to carry out similar constructions as in the proof of Theorem \ref{parts} in Section \ref{assembly}. This is why Corollary \ref{productproj} omits dimensions $m=2$, as the docking station construction only works for $S^{m+1}_k$ if $m+1\ge 4$. We expect that one should be able to construct metrics analogous to those of \cite[Lemma 1]{Sha} for $\mathbf{P}^n\times S^2$ using doubly warped Riemannian submersion metrics, thus extending Corollary \ref{productproj} to include the case $m=2$.


\subsection{Topological Constructions} 


In this section we prove that $\#_i (N_i^n\times S^m)$ can be realized as modified surgery on $S^{n-1}\times S^{m+1}$, specifically we claim the following. 

\begin{prop} \label{decomp} For any closed, oriented, and smooth manifolds $N_i^n$.
$$  \#_{i=1}^k (N_i^n \times S^m) \cong\left( S^{n-1} \times \left( S^{m+1}  \setminus \bigsqcup_{i=0}^{k}  D^{m+1} \right) \right)\cup_\partial \left(D^n\times S^m  \sqcup \bigsqcup_{i=1}^k (N_i^n\setminus D^n)\times S^m\right).$$
\end{prop}

\noindent Proposition \ref{decomp} follows from two lemmas: Lemma \ref{fix} which establishes Proposition \ref{decomp} when $k=1$, and Lemma \ref{handle} which establishes that performing modified surgery on trivially embedded spheres is equivalent to taking connected sums with $N^n\times S^m$. 


\begin{lemma} \label{fix} Let $N$ be a closed, oriented, and smooth manifold. 
$$ N^n\times S^m \cong \left( S^{n-1} \times \left( S^{m+1} \setminus\left(  D^{m+1}  \sqcup  D^{m+1}\right)\right)\right) \cup_\partial \left(\left(D^{n}\times S^m \sqcup ((N^{n}\setminus D^n)\times S^m)\right)\right).$$
\end{lemma}

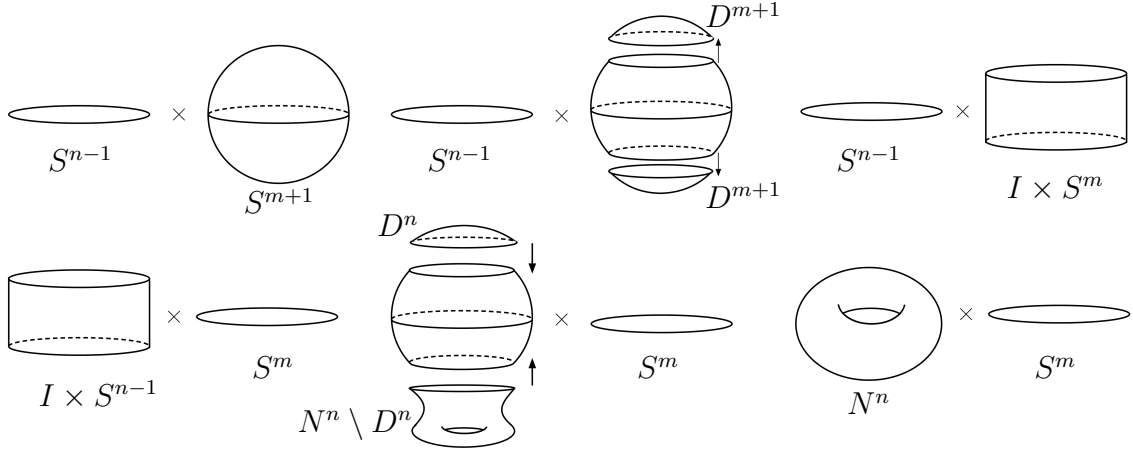
\begin{figure}
\centering
\begin{tikzpicture}[scale=.08]
\node (img) {\includegraphics[scale=0.08]{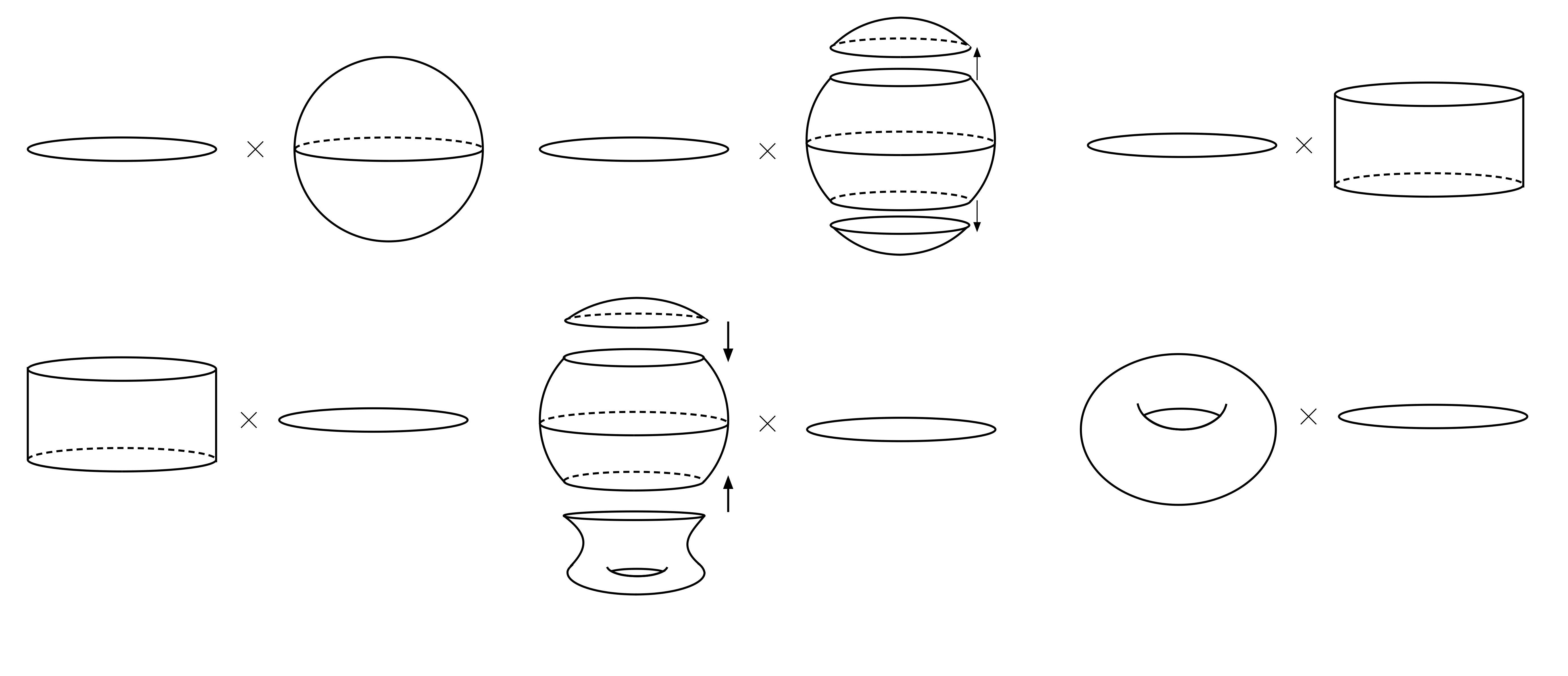}};
\node[below] at (-82,20) {$S^{n-1}$};
\node[below] at (-79,-19) {$I\times S^{n-1}$};
\node[below] at (-50,-15) {$S^m$};
\node[below] at (-49,13) {$S^{m+1}$};
\node[below] at (-19,20) {$S^{n-1}$};
\node[left] at (-24,5) {$D^n$};
\node[left] at (-25,-28) {$N^n\setminus D^n$};
\node[below] at (14,-15) {$S^m$};
\node[right] at (20,10) {$D^{m+1}$};
\node[right] at (20,40) {$D^{m+1}$};
\node[below] at (49,20) {$S^{n-1}$};
\node[below] at (80,15) {$I\times S^m$};
\node[below] at (80,-15) {$S^m$};
\node[below] at (49, -21) {$N^n$};
\end{tikzpicture}
 \caption{An illustrated proof of Lemma \ref{fix}.}
 \label{fig:fixit}
\end{figure}


\begin{proof} As illustrated in Figure \ref{fig:fixit}, we begin by noting that
$$S^{n-1} \times \left( S^{m+1} \setminus \left( D^{m+1}  \sqcup  D^{m+1}\right)\right) \cong S^{n-1} \times I \times S^m  \cong (S^n\setminus  (D^n\sqcup D^n))\times S^m.$$
Thus the righthand side of the equation in the claim becomes
$$ (S^n\setminus  (D^n\sqcup D^n))\times S^m \cup_\partial \left((D^{n} \sqcup ( N^{n} \setminus D^n))\times S^m\right).$$
Finally, notice that 
$$S^n\setminus  (D^n\sqcup D^n) \cup_\partial (D^n \sqcup( N^n\setminus D^n))\cong N^n .$$
\end{proof}



\begin{lemma} \label{handle} Let $N^n$ and $M^{n+m}$ be closed, oriented, and smooth manifolds. If\\ $f:S^{n-1}\times \mathring{D}^{m+1}\hookrightarrow M$ is a nullhomotopic embedding, then 
\begin{equation}\label{surgery} M^{n+m}\#(N^n\times S^m) \cong ( M^{n+m} \setminus \Ima  f ) \cup_{f|_\partial} ((D^n\setminus N^n) \times S^m).\end{equation}
\end{lemma} 


\begin{proof} 
As $f$ is nullhomotopic, we may isotope it so that there is an inclusion $\iota: D^{n+m}\hookrightarrow M^{n+m}$ of a small geodesic ball so that $f$ is realized as the composition
$$ \begin{tikzcd} S^{n-1}\times \mathring{D}^{m+1}\arrow{r}{} & S^{n-1}\times \mathring{D}^m \times (-\varepsilon ,\varepsilon) \arrow[hookrightarrow]{r}{(\nu, \ID) } &  S^{n+m-1}\times (-\varepsilon,\varepsilon)  \arrow[hookrightarrow]{r}{(\iota, \exp_\iota)} & M^{n+m}  \end{tikzcd},$$ 
where $\nu$ is the inclusion of a normal neighborhood of a great sphere as in the first part of Figure \ref{fig:isotoped}, and $(\iota,\exp_\iota)$ is the exponential map of a normal neighborhood for $\iota$. Note that $\Ima \iota \setminus \Ima f$ is an embedded $D^n\times D^m\hookrightarrow D^{n+m}$ as in the second part of Figure \ref{fig:isotoped}.


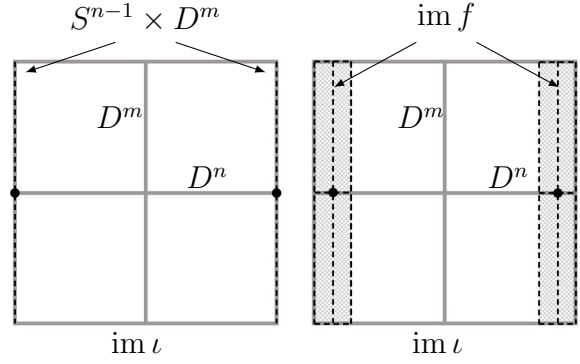
\begin{figure}
\centering
\begin{tikzpicture}[scale=.2]
\node (img) {\includegraphics[scale=0.2]{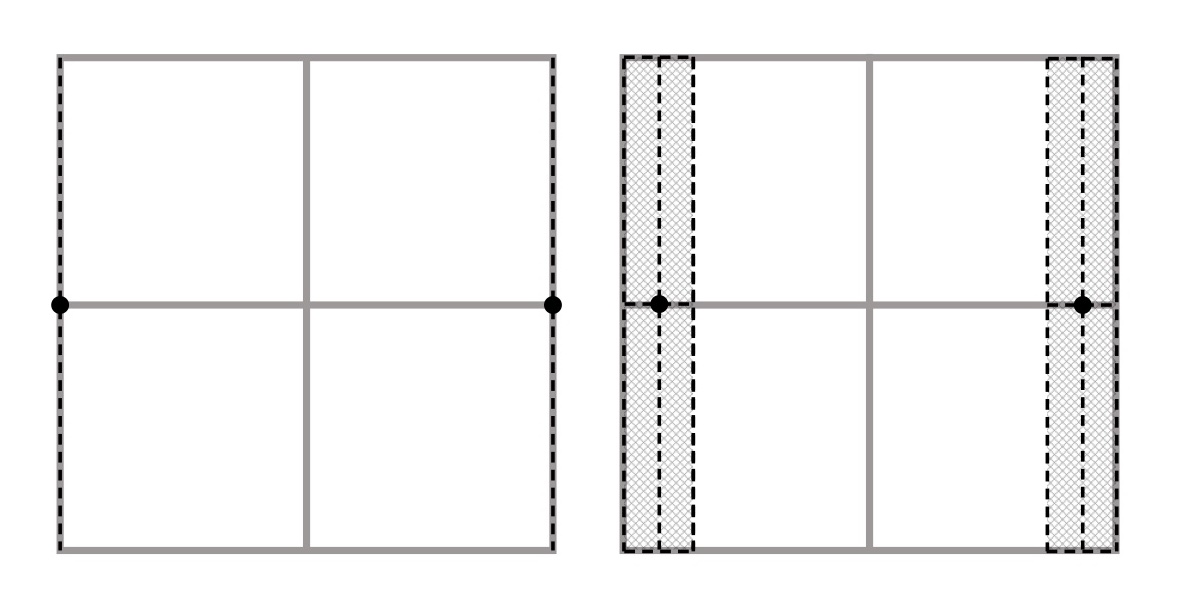}};

\node[right] at (-14,5) {$D^m$};
\node[right] at (-8,1) {$D^n$};
\node[above] at (-10,10) {$S^{n-1}\times D^m$};
\draw[->] (-8,10) to (-2,8);
\draw[->] (-12,10) to (-18,8);
\node[right] at (-13,-10) {$\Ima \iota$};

\node[right] at (6,5) {$D^m$};
\node[right] at (12,1) {$D^n$};
\node[above] at (10,10) {$\Ima f$};
\draw[->] (8,10) to (2.5,7);
\draw[->] (12,10) to (17.5,7);
\node[right] at (7,-10) {$\Ima \iota $};
\end{tikzpicture}
\caption{A schematic of the embeddings $\iota: D^{n+m}\hookrightarrow M^{n+m}$ and $f:S^{n-1}\times D^{m+1}\hookrightarrow M^{n+m}$.}
\label{fig:isotoped}
\end{figure}

It is explicit in the righthand side of equation (\ref{surgery}) that we are removing $\Ima f$ from $M^{n+m}$ and gluing in $(N^n\setminus D^n )\times S^m$. It is implicit on the lefthand side that we are removing $\Ima \iota$ from $M^{n+m}$ and gluing in $N^n\times S^m\setminus D^{n+m}$. To see that the resulting spaces are diffeomorphic, we will show that gluing $N^n\times S^m\setminus D^{n+m}$ into the deleted $D^{n+m}$ can be performed in two steps: the first step replaces $\Ima i\setminus \Ima f$ in the deleted $D^{n+m}$ and the second step glues in $(N^n\setminus D^n)\times S^m$. Once this is established, the two side of Equation (\ref{surgery}) are transparently equivalent.




\begin{figure}
\centering

\begin{tikzpicture}[scale=.12]
\node (img) {\includegraphics[scale=0.12]{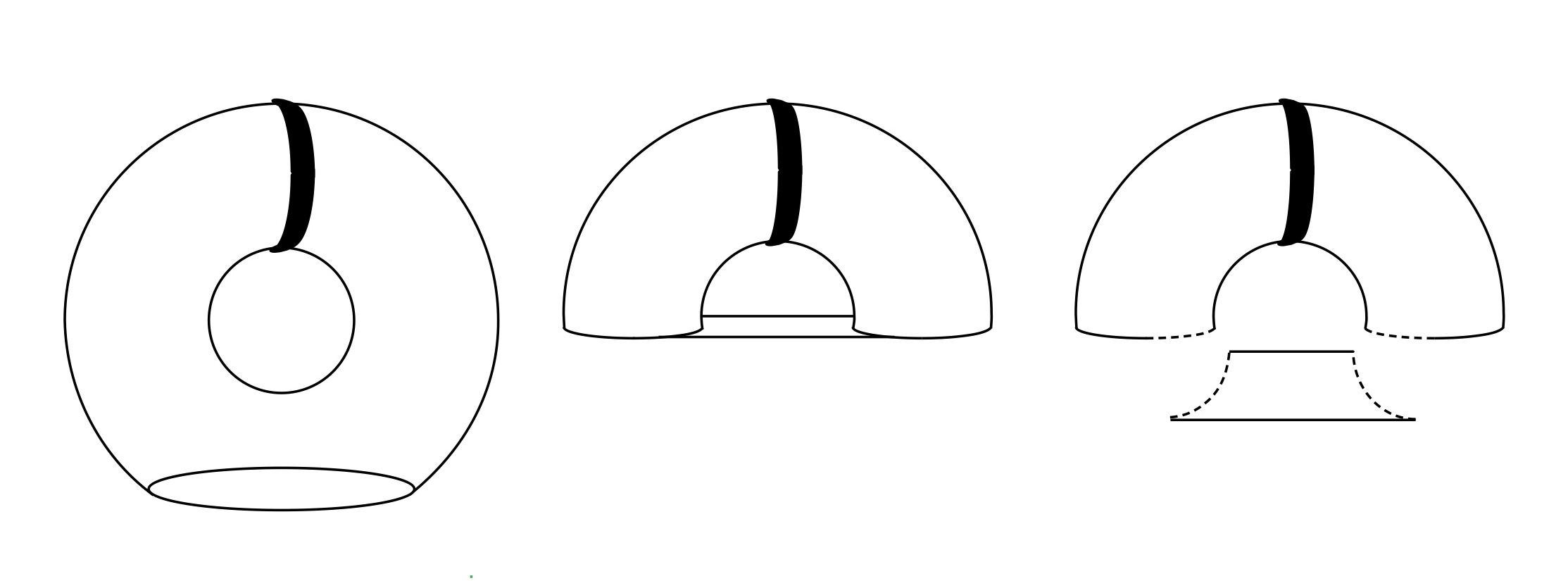}};
\node[below] at (0,-2) {$(N^n\setminus D^n)\times S^m $};
\node[below] at (4,-6) { $\cup D^n\times D^m $}; 
\node[below] at (25,-6) { $ D^n\times D^m $}; 
\node[below] at (25,15) { $ (N^n\setminus D^n)\times S^m $}; 
\node[below] at (-25,15) { $N^n\times S^m\setminus D^{n+m}$};
\end{tikzpicture}
\caption{Decomposing $N^n\times S^m\setminus D^{n+m}$ as $(N^n\setminus D^n)\times S^m \cup_{S^{n-1}\times D^m} D^n\times D^m$.}
\label{fig:lemma21}
\end{figure}
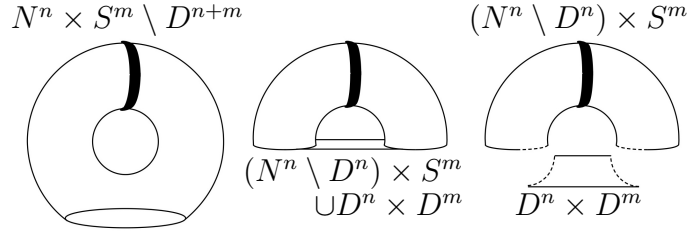

To achieve the two step identification, we must further analyze the space $N^n\times S^m \setminus D^{n+m}$. We claim that, as in figure \ref{fig:lemma21},
\begin{equation}\label{twopieces} N^n\times S^m \setminus D^{n+m} = (N^n\setminus D^n) \times S^m \cup_{S^{n-1}\times D^m} D^n\times D^m.\end{equation}
Note that that Equation (\ref{twopieces}) is transparently true when $N^n=S^n$ using the standard handle decomposition of $S^n\times S^m$ with four handles:
$$S^n\times S^m = \left( (D^{n+m}\cup_{S^{m-1}\times D^n} D^m\times D^n)  \cup_{S^{n-1}\times D^m} D^n\times D^m \right) \cup_{S^{n+m-1}} D^{n+m}.$$
Noting that $D^{n+m}\cup_{S^{m-1}\times D^n} D^m\times D^n \cong D^n\times S^m$, we see that
\begin{equation}\label{idk} S^n\times S^m = \left( D^n\times S^m  \cup_{S^{n-1}\times D^m} D^n\times D^m \right) \cup_{S^{n+m-1}} D^{n+m}.\end{equation}
Note that we may remove a small tubular neighborhood $D^n_\varepsilon \times S^m$ of $S^m$ from both sides of Equation (\ref{idk}) and glue in a $(N^n\setminus D^n) \times S^m$ (the dark band in Figure \ref{fig:lemma21}), resulting in
\begin{equation}\label{fin} N^n\times S^m = \left((N^n\setminus D^n)\times S^m  \cup_{S^{n-1}\times D^m} D^n\times D^m \right) \cup_{S^{n+m-1}} D^{n+m}.\end{equation}
Removing a $D^{n+m}$ from both sides of Equation (\ref{fin}) gives Equation (\ref{twopieces}).

To see the equivalence of $M^{n+m}\# (N^n\times S^m)$ and $(M^{n+m}\setminus \Ima f )\cup_\partial (N^n\setminus D^n) \times S^m$ consider Figure \ref{fig:gluing}. We start with $M^{n+m}\# (N^n\times S^m)$. Next we decompose $M^{n+m}\# (N^n\times S^m)$ as $M^{n+m}\setminus D^{n+m}$ and $N^n\times S^m\setminus D^{n+m}$ to be identified along their boundary. We can then decompose $N^n\times S^m\setminus D^{n+m}$ as in Equation (\ref{twopieces}) allowing us to identify the two pieces along their boundaries to $M^{n+m}\setminus D^{n+m}$ separately. We first attach $D^n\times D^m$ to $M^{n+m}\setminus D^{n+m}$ along $D^n\times S^{m-1}$ filling in $\Ima \iota\setminus \Ima f$ in the interior. This leaves us with $M^{n+m}\setminus \Ima f$. Thus in the final step when we attach $(N^n\setminus D^n)\times S^m$ we have constructed $(M^{n+m}\setminus \Ima f )\cup_\partial (N^n\setminus D^n) \times S^m$.



\end{proof}

\begin{figure}

\begin{tikzpicture}[scale=.095]
\node (img) {\includegraphics[scale=0.095]{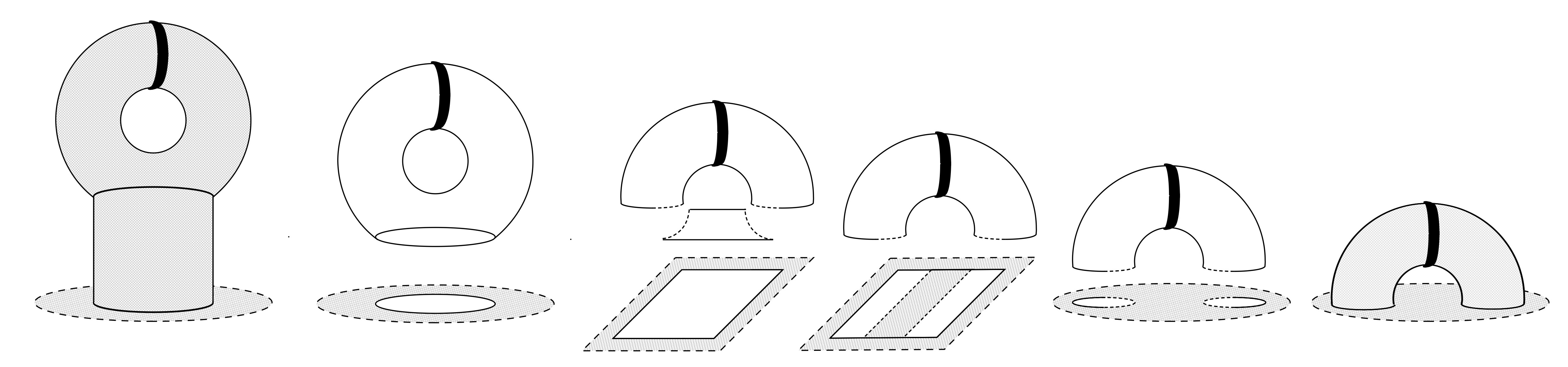}};
\tiny
\node[above] at (-70,-26) {$M^{n+m}\#(N^n\times S^m)$};
\node[above] at (-40,-26) {$M^{n+m}\setminus \Ima \iota$};
\node[below] at (-38,25) {$(N^n\times S^m)\setminus \Ima \iota$};
\node[above] at (-15,-26) {$M^{n+m}\setminus \Ima \iota$};
\node[below] at (-10,20) {$(N^n\setminus D^n)\times S^m$};
\node[below] at (0,16) {$ \cup {D}^n\times D^m $};
\node[above] at (10,-26) {$M^{n+m}\setminus \Ima \iota  $};
\node[below] at (15,-26) {$ \cup {D}^n\times D^m $};
\node[below] at (19,12) {$(N^n\setminus D^n)\times S^m$};
\node[above] at (40,-26) {$M^{n+m}\setminus \Ima f  $};
\node[below] at (44,8) {$(N^n\setminus D^n)\times S^m$};
\node[above] at (70,-26) {$M^{n+m}\setminus \Ima f  $};
\node[below] at (74,-26) {$\cup (N^n\setminus D^n)\times S^m)$};
\end{tikzpicture}

\caption{The equivalence of $M^{n+m}\# (N^n\times S^m)$ and $M^{n+m}\setminus \Ima f \cup_{f|_\partial} (N^n\setminus D^n) \times S^m$.}
\label{fig:gluing}
\end{figure}
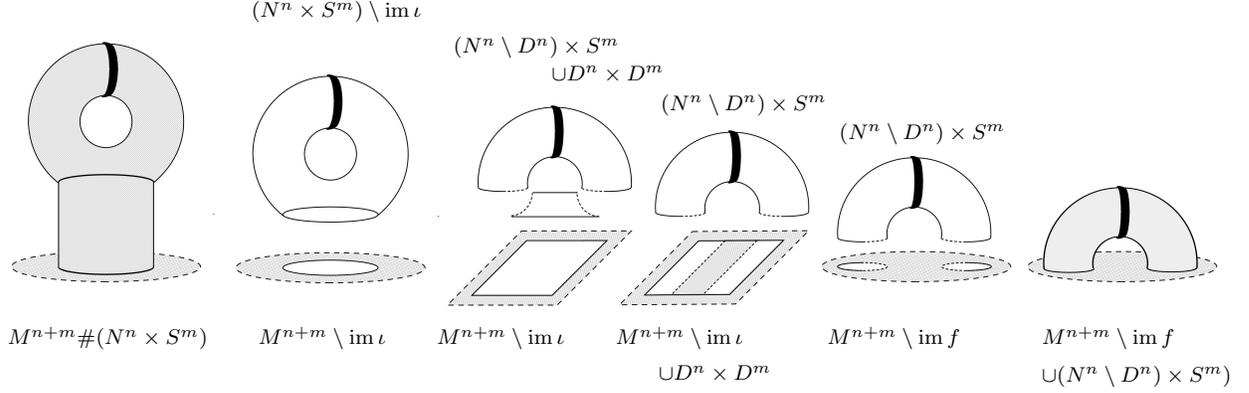


With Lemma \ref{fix} and \ref{handle} established, we can now prove our main surgery result. 

\begin{proof}[Proof of Proposition \ref{decomp}] As the $D^{m+1}$ on the righthand-side of the equation in the claim are disjoint, we may perform the modified surgeries in any order. Specifically, we may first perform surgery on two, and then perform surgery on the rest. Thus the righthand side of the equation becomes
$$\left(((S^{n-1} \times ( S^{m+1}\setminus  (D^{m+1}\sqcup D^{m+1}) )) \cup_\partial (D^n\times S^m \sqcup  ( N_1^n\setminus D^n) \times S^m ) )\setminus \left(\bigsqcup_{i=2}^k S^{n-1}\times D^{m+1}\right) \right) $$
$$ \text{\hspace{5in}} \cup_\partial \left( \bigsqcup_{i=2}^k (N_i^n\setminus D^n)\times S^m) \right).$$
By Lemma \ref{fix} this reduces to the following
$$\left( (N_1\times S^m) \setminus  \left(\bigsqcup_{i=2}^k S^{n-1}\times D^{m+1}\right) \right) \cup_\partial \left( \bigsqcup_{i=2}^k ( N_i^n\setminus D^n)\times S^m) \right).$$

We now would like to claim that Lemma \ref{handle} applied $(k-1)$ times proves the claim. This is not immediately obvious. If each of the implied embeddings $S^{n-1}\times D^{m+1}\hookrightarrow N^n\times S^m$ were nullhomotopic, Lemma \ref{handle} could be applied to any one of them. But in order to guarantee that we can apply Lemma \ref{handle} in succession to each embedding, we must show that they remain nullhomotopic after performing the other surgeries. It suffices to show that the image of all the nullhomotopies are disjoint because then each can be isotoped to disjoint geodesic balls, for which Lemma \ref{handle} may be applied in succession.


We claim that these nullhomotopies exist and are disjoint. Begin by retracting the embedded $S^{n-1}\times D^{m+1}$ onto its embedded core $S^{n-1}$, so that we need only show the disjoint embeddings of $S^{n-1}$ are nullhomotopic and that the nullhomotopies are disjoint. Recall that we begin with $(k-1)$ embeddings $f_j:S^{n-1} \hookrightarrow S^{n-1}\times S^{m-1}$ where the image is $S^{n-1}\times \{p_j\}$ for distinct $p_j$, which after isotoping $f_j$ we may assume all the $p_j$ lie in a great circle disjoint from the embedded $D^{m+1}\sqcup D^{m+1}$ of Lemma \ref{fix} . One can trace the embeddings $f_j$ through the construction of Lemma \ref{fix} (figure \ref{fig:fixit}) to give embeddings $h_j:S^{n-1}\hookrightarrow N^n\times S^m$ pictured in figure \ref{fig:impliedembed}. The image of $h_j$ is an $S^{n-1}\times\{p_j\}$ which bounds the disk $D^n\times \{p_j\}\subseteq D^n\times S^m$, where this is the $D^n\times S^m$ being attached on the righthand side in Lemma \ref{fix}. The disjoint $D^{n+1}\times \{p_j\}$ provide disjoint nullhomotopies of the $h_j$. 
\end{proof}

\begin{figure}
\begin{tikzpicture}[scale=.06]
\node (img) {\includegraphics[scale=0.06]{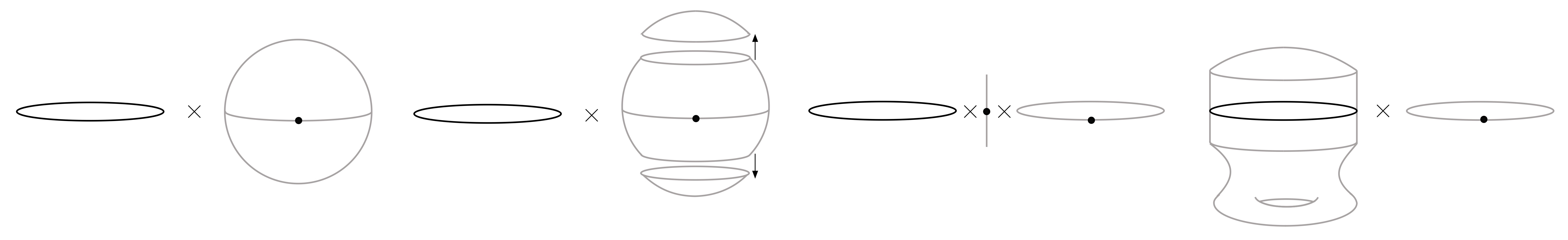}};
\end{tikzpicture}
\caption{The implied embedding $S^{n-1}\times D^{m+1}\hookrightarrow N^n\times S^m$ given by the black subset.}
\label{fig:impliedembed}
\end{figure}


\subsection{Metric Construction}

Proposition \ref{decomp} decomposes $\#_i(N_i^n\times S^m)$ as a boundary union of two Riemannian manifolds, we seek to construct Ricci positive metrics on each with boundary conditions that allow us to use Lemma \ref{glue}. 

\begin{prop}\label{metric} Let $n>2$ and $m\ge 3$. If there exists Ricci positive metrics on $M^n_i\setminus D^n$ with round, convex boundaries, then the following manifold admits a metric with positive Ricci curvature.
$$\#_{i=1}^k (N^n_i \times S^m).$$
\end{prop} 

\begin{proof} We will use the specific construction of $\#_{i=1}^k (N_i^n\times S^m)$ provided in Proposition \ref{decomp}. This theorem decomposes $\#_{i=1}^k (N_i^n\times S^m)$ as the boundary union of two smooth manifolds: 
$$\text{DS}:= S^{n-1}\times \left(S^{m+1} \setminus \left( \bigsqcup_{i=0}^k D^{m+1}\right)\right)\text{ and } \text{Cores}:=  \bigsqcup_{i=0}^k ( N_i^n\setminus D^n)\times S^m,$$
where we have set $N_0^n=S^n$. Our approach is to construct metrics on DS and Cores using Proposition \ref{docking} and our assumptions about $M_i^n$ respectively, so that they have positive Ricci curvature, isometric boundaries, and principal curvatures that allow Lemma \ref{glue} to apply. 

We have assumed that there are core metrics on $N_i^n$ for $i>0$, but note that $N_0^n= S^n$ also admits a core metric, specifically $ds_{n}^2$ restricted to a geodesic ball of radius $r< \pi/2 $. We will assume therefore that the $N_i^n$ all admit core metrics $g_i$, and we will use these core metrics to define a metric on $(N_i^n\setminus D^n)\times S^m$. Consider first the product metric $g_i+ \rho ds_m^2$ on $(N_i^n\setminus D^n)\times S^m$, clearly this metric has positive Ricci curvature. However, the second fundamental form of the boundary restricted to $TS^m$ is zero. To apply Lemma \ref{glue}, we will need the boundary to be convex. This can be achieved by bending the boundary slightly, and if the bend is slight enough, this can be done without upsetting positive Ricci curvature. 

Take a collar neighborhood $N_\xi S^{n-1} = (-\xi ,0]\times S^{n-1}$ inside $N^n_i\setminus D^n$. In these coordinates, the metric splits as $g_i=dt^2+g_i(t)$ where $g_i(t)$ are metrics on $S^{n-1}$ for $t\in (-\xi , 0]$. We may define metrics
$$\tilde{g}_i = \begin{cases} dt^2+g_i(t)+ f^2(t) ds_m^2 & x\in N_\xi S^{n-1} \\ g_i+\rho^2 ds_m^2 & x\in (N_i^n\setminus N_\xi S^{n-1}).\end{cases}$$
where $f(t)$ is any smooth functions that is constant for $ -\xi < t <-\xi_0$ for some $\xi_0$. If $\widetilde{\2}_i$ and $\2_i$ is second fundamental form of boundary with respect to $\tilde{g}_i$ and $g_i+\rho^2ds_m^2$ respectively, then one can easily check that $\widetilde{\2}_i|_{S^m} = (f'(t)/f(t)) \tilde{g}_i|_{S^m}$ and $\widetilde{\2}_i|_{S^{n-1}} = \2_i|_{S^{n-1}}$. Thus if $f'(0)>0$, the boundary is convex. Moreover, there exists an $\varepsilon>0$ there is $\delta>0$ such that if $||f(t)-\rho||_{C^2}<\delta$ then $||(g_i+\rho^2 ds_m^2) -\tilde{g}_i||_{C^2}<\varepsilon$. Thus for adequately chosen $f(t)$, $\tilde{g}_i$ will have positive Ricci curvature. 




Thus, there exists a Ricci positive metric $\tilde{g}_i$ on $(N^n\setminus D^n)\times S^m$ with convex boundary isometric to $\kappa^2 ds_n^2 + \rho^2 ds_m^2$. Let $\nu_i=\inf_{v\in S(T(S^n\times S^m))} \widetilde{\2}_i(v,v)$. We have shown that $\nu_i>0$. Pick a number $0<s<1$, such that ${\nu_i}/{s}>1$ for all $i$. Notice that $((N^n_i\setminus D^n)\times S^m),s^2\tilde{g}_i)$ has positive Ricci curvature, boundary isometric to $(S^{n-1}\times S^m, (s\kappa)^2 ds_{n-1}^2 + (s\rho)^2 ds_m^2)$, and with principal curvatures of the boundary all greater than $1$. Define $g_\text{cores}$ to be $s^2\tilde{g}_i$ on each component of Cores. 

We now turn to defining a metric on DS. By Proposition \ref{docking}, there exists a metric $g_\text{docking}$ on $S^{m+1}_k$ with boundary components all isometric to $(S^m, (s\rho)^2 ds_m^2)$ and principal curvatures identically $-1$. Define $g_{\text{DS}}=(s\kappa)^2 ds_{n-1}^2 + g_\text{docking}$. We see that the second fundamental form of the boundary $\2$ restricted to $TS^{n-1}$ is zero, so that the principal curvatures of the boundary of DS with respect to $g_\text{DS}$ are all bounded below by $-1$. Clearly the boundary of $(\text{DS},g_\text{DS})$ is isometric to the boundary of $(\text{Cores}, g_\text{cores})$.  As all the principal curvatures of the boundary of $(\text{Cores}, g_\text{cores})$ are greater than 1, Lemma \ref{glue} applies. Thus there is a smooth metric on $\text{DS} \cup_\partial \text{Cores}$ with positive Ricci curvature, which by Proposition \ref{decomp}, is diffeomorphic to $\#_{i=1}^k (N_i^n\times S^m)$. 
\end{proof}

In particular, we have shown that complex, quaternionic, and octonionic projective spaces admits core metrics. Thus the Corollary \ref{productproj} is immediate. 

\begin{proof}[Proof of Corollary \ref{productproj}] By Theorem \ref{A} and the observation that $S^n$ admits a core metric in the proof of Proposition \ref{metric}, the spaces $S^n$, $\CP^n$, $\HP^n$, and $\OP^2$ all admit core metrics. By Proposition \ref{metric}, the claim follows. 
\end{proof}



\section*{Acknowledgements}

I would like to thank my advisor, Boris Botvinnik, for his patience and writing advice. Thanks to David Wraith for for his diligent notes on my early drafts. And I would like to thank my colleague Demetre Kazaras for motivating me early on in this project.


\bibliographystyle{elsarticle-num} 
\bibliography{references}





\end{document}